\documentclass[12pt]{amsart}
\usepackage{amsthm}
\usepackage{amsfonts}
\usepackage{graphicx}
\usepackage{fullpage}

\newcommand{\C}{{\mathbb{C}}}

\newtheorem{theorem}{Theorem}[section]
\newtheorem{corollary}[theorem]{Corollary}
\newtheorem{lemma}[theorem]{Lemma}
\newtheorem{proposition}[theorem]{Proposition}

\newtheorem*{thm}{Theorem}

\newtheorem{definition}[theorem]{Definition}

\numberwithin{figure}{section}
\numberwithin{theorem}{section}

\begin{document}

\author{Heather M. Russell}
\address{Department of Mathematics, Louisiana State University, Baton Rouge, 
LA 70803}
\email{hrussell@math.lsu.edu}
\title{A topological construction for all two-row Springer varieties}

\begin{abstract}
Springer varieties appear in both geometric representation theory and knot theory. Motivated by knot theory and categorification Khovanov provides a topological construction of $(n/2, n/2)$ Springer varieties. We extend Khovanov's construction to all two-row Springer varieties. Using the combinatorial and diagrammatic properties of this construction we provide a particularly useful homology basis and construct the Springer representation using this basis. We also provide a skein-theoretic formulation of the representation in this case.
\end{abstract}

\maketitle

\section{Introduction}
Springer varieties (or Springer fibers) are certain subvarieties of the variety of full flags in $\mathbb{C}^n$. Given a partition $\lambda$ of the number $n$ the Springer variety $\mathcal{S}_{\lambda}$ is the collection of full flags in $\mathbb{C}^{n}$ fixed by a nilpotent linear operator with Jordan blocks given by $\lambda$. Springer varieties were first introduced by Springer who constructed irreducible representations of the symmetric group on their top nonzero cohomology classes \cite{Springer}. This remarkable construction has motivated the study of these varieties by geometric representation theorists (for example see \cite{F, Hotta, KL}).

Springer varieties are also appearing with increasing frequency in the literature on knot homologies and categorification. For $n$ even Khovanov constructs a functor-valued invariant of tangles using an arc algebra $H^{n/2}$ with center isomorphic to the cohomology of the $(n/2, n/2)$ Springer variety \cite{KTan}. In proving this isomorphism he provides a topological construction of  the $(n/2, n/2)$ Springer variety as a subspace of a product of spheres \cite{K}. Springer varieties also appear in Cautis-Kamnitzer's knot homology via derived categories of coherent sheaves \cite{CK}, in Seidel-Smith's link invariant from the symplectic geometry of nilpotent slices \cite{SS}, and in work of Stroppel including \cite{SW}.

The structure of Springer varieties is not well understood. In particular for general classes of Springer varieties the topology of individual components and the interaction of those components in not known. Finding topological models for Springer varieties has the potential to aid geometric representation theory as well as deepen our understanding of the connections between Springer varieties and knot theory.

In previous work we use Khovanov's topological construction to prove an isomorphism between the homology of the $(n/2,n/2)$ Springer variety and the Bar-Natan skein module of the solid torus with boundary web $n$ copies of the longitude \cite{R}. We also use this construction to give a completely  explicit and combinatorial construction of Springer's representation on this class of Springer varieties \cite{RT}.

The main result of this paper is an extension of Khovanov's construction of $(n/2, n/2)$ Springer varieties to all Springer varieties $\mathcal{S}_{\lambda}$ where $\lambda$ is a two element partition of the number $n$ (not necessarily even). This construction can be found in Section 2. The proof that our construction is homeomorphic to the Springer variety follows the structure and approach of \cite[Appendix]{RT}. 

In Section 3 we analyze the intersection of irreducible components of these varieties and build an exact sequence on homology.  In order to construct this sequence we generalize many of the results found in \cite{K}. We use two results from \cite{K}. For completeness we include these proofs in an Appendix to this paper.

Sections 4 and 5 provide two applications of our topological construction of these Springer varieties. The first is a particularly nice diagrammatic homology basis. The second is an extension of the construction of the Springer representation from \cite{RT} to this more general class.  

In response to a question of Stephan Wehrli the following theorem in Section 5 provides a skein-theoretic formulation of the Springer representation in the two-row case.
\begin{thm}
Given a diagrammatic homology generator $M\in H_*(\mathcal{S}_{n-k,k})$ and $\sigma\in S_n$ glue a flattened braid corresponding to $\sigma$ to the bottom of $M$ forming $M'$. 
Then the Springer action $\sigma\cdot M$ is equal to $s(M')$ where $s$ is defined as follows.
\begin{itemize}
\item{$s\left( \raisebox{-.13in}{\includegraphics[width=.3in]{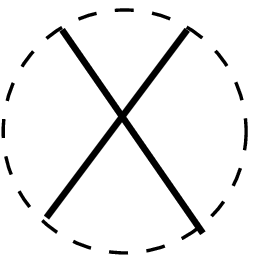}}\right) = s\left( \raisebox{-.13in}{\includegraphics[width=.3in]{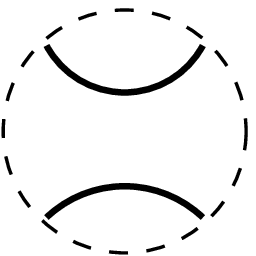}}\right) + s \left( \raisebox{-.13in}{\includegraphics[width=.3in]{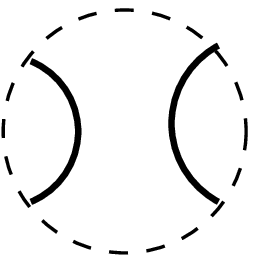}}\right)$}
\item{$s\left( \raisebox{-.15in}{\includegraphics[width=.3in]{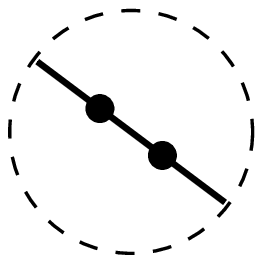}}\right) = 0$}
\item{ $s\left( M' \sqcup \raisebox{-.1in}{\includegraphics[width=.3in]{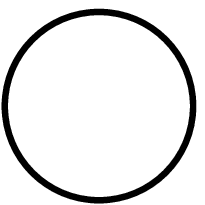}} \right) = s(-2M')$}
\item{$s\left( M' \sqcup \raisebox{-.1in}{\includegraphics[width=.3in]{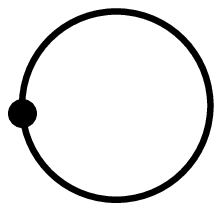}} \right) = s(M')$}
\end{itemize}
\end{thm}

Many thanks to Charlie Frohman and Julianna Tymoczko for their continued support and guidance as well as many helpful conversations related to this work. A significant part of this paper was written in Spring 2010 while attending the Homology Theories of Knots and Links program at MSRI. Thank you to MSRI for this opportunity. The author was also partially supported by NSF VIGRE grant DMS 0739382. 

\section{Extending Khovanov's topological construction}

For $n$ even Khovanov constructs a topological space with cohomology isomorphic to that of the $(n/2,n/2)$ Springer variety and conjectures that the two are actually homeomorphic \cite[Conjecture 1]{K}.  The main ideas behind a proof of this fact can be found in \cite{CK}, and detailed proofs are given in the Appendix of \cite{RT} and independently in \cite{W}. In this section we generalize Khovanov's construction to all two-row Springer varieties. The proof that our construction is homeomorphic to the Springer variety follows that of \cite[Appendix]{RT}.

Let $n\geq1$ be some positive integer.  A complete flag in $\mathbb{C}^n$, denoted by  $V_{\bullet}$, is a collection of nested subspaces
$$V_1\subset V_2\subset \cdots \subset V_{n-1} \subset V_n$$
such that the complex dimension of $V_i$ is $i$. The collection of all such objects is the algebraic variety $\mathcal{F}_n$. Partial flags have the same nesting property but are not required to have subsets in every intermediate dimension.

Let $\lambda = (\lambda_1\geq \lambda_2 \geq \cdots \geq \lambda_k)$ be a partition of the number $n$. Let $\Gamma: \C^{n} \rightarrow \C^n$ be a nilpotent linear operator with Jordan blocks of sizes specified by the partition $\lambda$. Then we have the following definition.
\begin{definition}
The Springer variety associated to the partition $\lambda$ is
$$\mathcal{S}_{\lambda} = \{ V_{\bullet} \in \mathcal{F}_n : \Gamma V_i \subseteq V_i \text{ for all } i\}.$$ 
\end{definition}

We focus on Springer varieties associated to two-element partitions which we call two-row Springer varieties. Note that every two-element partition of $n$ can be written as $(n-k,k)$ for some positive integer $0\leq k\leq \lfloor \frac{n}{2} \rfloor$.    We call $\mathcal{S}_{n-k,k}$ the $(n-k,k)$ Springer variety.

The irreducible components of Springer varieties are indexed by standard Young tableaux \cite{S, V}. In the two-row case, these tableaux are in one-to-one correspondence with noncrossing matchings. The definition below generalizes the one found in \cite[pg 3]{K} and can be found in \cite[Definition 1.2]{SW}.
\begin{definition}
Consider $n$ vertices evenly spaced along a horizontal line. A noncrossing matching of type $(n-k,k)$ is a nonintersecting arrangement of $k$ arcs and $n-2k$ rays incident on the $n$ vertices lying above the horizontal line. We assume that the rays are ``infinitely high" so that arcs cannot cross rays. Figure \ref{ncm} has an example. 
\end{definition}

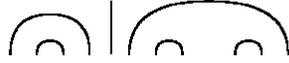
\begin{figure}[h]
\begin{picture}(100,20)(0,0)
\qbezier(10, 0)(10,5)(5,5)
\qbezier(5,5)(0,5)(0,0)
\qbezier(20,0)(20,15)(5,15)
\qbezier(5,15)(-10,15)(-10,0)
\put(28,0){\line(0,1){20}}
\qbezier(35,0)(35,20)(65,20)
\qbezier(65,20)(95,20)(95,0)
\qbezier(55, 0)(55,5)(50,5)
\qbezier(50,5)(45,5)(45,0)
\qbezier(85, 0)(85,5)(80,5)
\qbezier(80,5)(75,5)(75,0)
\end{picture}
\caption{A noncrossing matching of type $(6,5)$}\label{ncm}
\end{figure}

Matchings will be denoted by lowercase letters. For clarity we will make these bold throughout. Let $B^{n-k,k}$ be the set of all noncrossing matchings of type $(n-k,k)$. 
Given $\bf{a}$ $\in B^{n-k,k}$ write $(i,j)\in \bf{a}$ if $\bf{a}$ has an arc connecting vertices $i$ and $j$. Write $(i)\in \bf{a}$ if $\bf{a}$ has a ray incident on the vertex $i$.

Let $N >2n$ be a large fixed integer. Let $X:\mathbb{C}^{2N} \rightarrow \mathbb{C}^{2N}$ be a nilpotent linear operator with two Jordan blocks of size $N$. Let $\{e_1, \ldots, e_N, f_1, \ldots, f_N\}$ be an orthonormal basis 

for $\mathbb{C}^{2N}$ with the property that $Xe_i = e_{i-1}$ and $Xf_i = f_{i-1}$. Here we define $e_{-1} = f_{-1} = 0$.

Consider the following variety $Y_n$ of partial flags in $\mathbb{C}^{2N}$.
$$Y_n = \{ V_{\bullet} = V_1\subset \cdots \subset  V_n : \textup{dim}_{\mathbb{C}}(V_i) = i \textup{ and } XV_i \subseteq V_{i-1}\} $$ From now on flags will be written $(V_1, \ldots, V_n)$ to save space.
For $0\leq k\leq \lfloor \frac{n}{2} \rfloor$, let $$V_{n-k,k} =  \langle e_1, \ldots , e_{n-k}, f_1, \ldots , f_k\rangle$$ and define $ Y_{n-k,k} = \{V_{\bullet} \in Y_n: V_n = V_{n-k,k} \}$.
Since $Y_{n-k,k}$ is the set of all complete flags on $V_{n-k,k}$ fixed by $X$, it is diffeomorphic to $\mathcal{S}_{n-k,k}$.

For $1\leq i\leq n$ define the subvariety $Z_n^i$ of $Y_n$ to be $Z_n^i = \{ V_{\bullet} \in Y_n: V_{i+1} = X^{-1}V_{i-1}\}$.
Define the map
\begin{eqnarray*}
\hspace{2in} Z_n^i &\stackrel{q}{\longrightarrow}&  Y_{n-2} \\
V_{\bullet} = (V_1, \ldots V_n) &\mapsto& V_{\bullet}' = (V_1,\ldots , V_{i-1},XV_{i+2} ,\ldots , XV_n)
\end{eqnarray*} 
This map is a $\mathbb{P}^1$ bundle \cite[pg 5]{CK}. 

Let $p = (0,0,1)$ be the north pole of the standard unit two-sphere $S^2$ embedded in $\mathbb{R}^3$. Let $-p = (0,0,-1)$ be the south pole in $S^2$ where $-p$ is understood as the antipodal map. Given $\bf{a}$ $\in B^{n-k,k}$ define the following subspaces of $(S^2)^n$.
\begin{itemize}
\item{$S_{{\bf a},n-k,k} = \{ (x_1, \ldots, x_n) \in (S^2)^n : x_i = x_j \textup{ if } (i,j)\in {\bf a} \textup{ and } x_i = (-1)^{i}p \textup{ if } (i)\in {\bf a} \}$}
\item{$S_{{\bf a},n-k,k}' =\{ (x_1, \ldots, x_n) \in (S^2)^n : x_i = -x_j \textup{ if } (i,j)\in {\bf a} \textup{ and } x_i = p \textup{ if } (i)\in {\bf a} \} $}
\end{itemize}
Taking unions over all ${\bf a}\in B^{n-k,k}$ define
$$X_{n-k,k} = \bigcup_{{\bf a}} S_{{\bf a},n-k,k} \textup {\hspace{.3in} and \hspace{.3in}} X_{n-k,k}' = \bigcup_{{\bf a}} S_{{\bf a},n-k,k}'.$$

The space $X_{n-k,k}$ is a generalization of Khovanov's construction of the $(n/2,n/2)$ Springer variety \cite[pg 4]{K}. Indeed in the $(n/2, n/2)$ case matchings have no rays, so each component is built by identifying coordinates pairwise as prescribed by the arcs of the associated matching. 

Using the basis $\{e_1, \ldots, e_N,f_1, \ldots, f_N\}$ for $\C^{2N}$ write points in $\mathbb{P}^{2N-1}$ as $\left\langle \sum a_ie_i+ \sum b_if_i \right\rangle$ for $a_i,b_i\in \mathbb{C}$. For $\mathbb{P}^1$ write $e_1 = e$ and $f_1 = f$. 
Define the map $C: \mathbb{P}^{2N-1} \rightarrow \mathbb{P}^1$ by
$$\left\langle \sum a_i e_i + \sum b_i f_i \right\rangle \mapsto \left\langle \left(\sum a_i\right) e + \left( \sum b_i\right)  f\right\rangle .$$ For each $V_{\bullet} \in Y_n$ define lines $L_1, \ldots, L_n$ by $V_i = V_{i-1} \oplus L_i$ and $L_i \perp V_{i-1}$.  Then  \cite[Theorem 2.1]{CK} proves

\begin{proposition}\label{CKProp}
The map $\ell: Y_n \rightarrow (\mathbb{P}^1)^n$ defined by
\[V_{\bullet} \mapsto (C(L_1), C(L_2), \ldots, C(L_n))\]
is a diffeomorphism. Furthermore the image of $Z_n^i$ under the diffeomorphism $\ell$ is exactly the elements in $(\mathbb{P}^1)^n$ 
satisfying $C(L_i) = -C(L_{i+1})$.
\end{proposition}

Let $s: S^2-\{ p \}  \rightarrow \C$ be stereographic projection, and let $\varphi: \mathbb{P}^1 -\{ \langle e \rangle \} \rightarrow \C$ be defined by $\varphi(\langle xe + yf \rangle) = x/y.$ Let $\Phi: \mathbb{P}^1 \rightarrow S^2$ be the diffeomorphism defined by
\[\begin{array}{rcl}
	\Phi(\langle xe+yf\rangle) &=& \left\{ \begin{array}{l} s^{-1}\circ \varphi (\langle xe+yf \rangle) \textup{ if }\langle xe+yf \rangle \neq \langle e \rangle \textup{ and } \\
						     (0,0,1) \textup{ if } \langle xe+yf \rangle = \langle e \rangle \end{array} \right. \end{array}\]
						     
Define the diffeomorphism $\widetilde{\Phi}:  (\mathbb{P}^1)^n \rightarrow (S^2)^n$ to be $\widetilde{\Phi}(x_1, \ldots, x_n) = (\Phi(x_1), \ldots , \Phi(x_n))$. Given ${\bf a} \in B^{n-k,k}$ define $C_{{\bf a},n-k,k}$ to be the preimage $\ell^{-1}\circ \widetilde{\Phi}^{-1}(S_{{\bf a},n-k,k}')$. 

For $k>0$ let ${\bf a} \in B^{n-k,k}$ such that $(i,i+1) \in {\bf a}$, and let ${\bf a'} \in B^{n-k-1, k-1}$ be the noncrossing matching obtained from ${\bf a}$ by erasing the arc $(i,i+1)$. If $f: \{1,2,\ldots, n\} \rightarrow \{1,2,\ldots,n-2\}$ is the map
\[\begin{array}{rcl}
	f(j) &=& \left\{ \begin{array}{l} j \textup{ if } j \leq i-1 \textup{ and } \\
						     j-2 \textup{ if } j \geq i+2\end{array} \right. \end{array}\]
then ${\bf a'}$ is comprised of the set of arcs $\{ (f(j), f(j')): (j,j') \in {\bf a} \textup{ and } (j,j') \neq (i,i+1)\}$ and the set of rays $\{(f(j)): (j)\in {\bf a}\}$.  Define the projection map $q': S'_{{\bf a},n-k,k} \rightarrow S'_{{\bf a'},n-k-1, k-1}$ by $q'(x_1, \ldots, x_n) = (x_1,  \ldots, \widehat{x_i}, \widehat{x_{i+1}}, \ldots, x_n)$ where $\widehat{x_j}$ omits the $j^{th}$ coordinate.

\begin{lemma}\label{erase}
For noncrossing matchings ${\bf a}$ and ${\bf a'}$ as above there is a commutative diagram
\[\begin{array}{rcl}
C_{{\bf a},n-k,k} & \stackrel{\widetilde{\Phi} \circ \ell}{\longrightarrow} & S'_{{\bf a},n-k,k} \\
\textup{\tiny{$q$}} \downarrow & & \downarrow {\textup{\tiny{$q'$}}} \\
Y_{n-2} &  \stackrel{\widetilde{\Phi} \circ \ell}{\longrightarrow} & S'_{{\bf a'},n-k-1,k-1}\end{array}\] and the image $q(C_{{\bf a},n-k,k})$ is $C_{{\bf a'},n-k-1,k-1}$.
\end{lemma}
\begin{proof}
We provide a short proof here. For more details see \cite[Lemma 5.2]{RT} which contains an identical argument.  

Since coordinates $x_i$ and $x_{i+1}$ in $S'_{{\bf a},n-k,k}$ are antipodes, Proposition \ref{CKProp} allows us to conclude $C_{{\bf a},n-k,k}\subseteq Z_n^i$. Let $V_{\bullet}\in Z_n^i$. By the definition of $Z_n^i$  we know that $X^{-1}V_{i-1} = V_{i+1}$ and thus $\textup{ker }X\subseteq V_{i+1}$. It follows that $L_j$ is spanned by $\sum_{i\geq2} a_ie_i+b_if_i$  and $XL_j\perp XV_{j-1}$ for all $j\geq i+2.$ This shows that if $V_{\bullet}\in C_{{\bf a},n-k,k}$  then $ (\widetilde{\Phi}\circ \ell) (q(V_{\bullet})) = q'((\widetilde{\Phi}\circ \ell)(V_{\bullet}))$. By commutativity of the diagram, we conclude that $q(C_{{\bf a},n-k,k})$ is $C_{{\bf a'}, n-k-1, k-1}.$
\end{proof}

\begin{lemma}\label{nk}
 The union $\bigcup_{{\bf a}\in B^{n-k,k}} C_{{\bf a},n-k,k}$ is  equal to $Y_{n-k,k}$. The $C_{{\bf a},n-k,k}$ are the irreducible components of $Y_{n-k,k}$.
\end{lemma}

\begin{proof}
Consider the partition $(n,0)$. The unique noncrossing matching ${\bf a}\in B^{n,0}$ has $n$ rays and no arcs. Then $X'_{n,0} = S_{{\bf a},n,0}' = \{ (p,\ldots, p)\in (S^2)^n \}$ and $$\widetilde{\Phi}^{-1}(S_{{\bf a},n,0}') = \{(\langle e\rangle, \ldots , \langle e\rangle)\in (\mathbb{P}^1)^n\}.$$ We want to find $V_{\bullet}\in Y_n$ such that $\ell(V_{\bullet}) =  (\langle e\rangle, \ldots , \langle e\rangle)$. 

Since $C(L_1) = \langle e \rangle$ we must have $L_1 = V_1 = \langle e_1 \rangle$. Inductively assume that $V_{i-1} = \langle e_1, \ldots , e_{i-1} \rangle.$ We must have $L_i\perp V_{i-1}$ and $XL_i\subset V_{i-1}$ so $L_i = \langle x_{i}e_i + y_1f_1 \rangle$. Since $C(L_i) = \langle e \rangle$ this forces $y_1 = 0$. This shows
\begin{eqnarray*}
(\ell^{-1}\circ \widetilde{\Phi}^{-1})(S'_{{\bf a},n,0}) &=& C_{{\bf a},n,0}\\
 &=& \{ (<e_1>, <e_1, e_2>, \ldots , <e_1, \ldots, e_n>) \} \\
&=& \{ (V_{1,0}, \ldots, V_{n,0})\}\\
&=& \{ V_{\bullet} \in Y_{n}: XV_i\subset V_i \textup{ and } V_n = V_{n,0} \}\\
&=& Y_{n,0}.
\end{eqnarray*}
Thus we have proven $C_{{\bf a},n,0}$ is diffeomorphic to $Y_{n,0}$ for all $n$. There is a unique noncrossing matching ${\bf a}\in B^{1,1}$ namely the one with arc $(1,2)$. From \cite[Lemma 5.3]{RT} $C_{{\bf a},1,1}$ is diffeomorphic to $Y_{1,1}$. 

We proceed by induction assuming the statement is true for $(n-k-1, k-1)$ in order to prove it for $(n-k,k)$. 
Assume that the claim holds for $(n-k-1, k-1)$ where $k\geq 1$. Let ${\bf a}\in B^{n-k,k}$ be a matching of type $(n-k,k)$. Because ${\bf a}$ is noncrossing, it necessarily has an arc of the form $(i,i+1)$. Let ${\bf a'}\in B^{n-k-1,k-1}$ be the matching of type $(n-k-1, k-1)$ obtained by erasing the arc $(i,i+1)$ in ${\bf a}$.  Then we have the $\mathbb{P}^1$ bundle  $C_{{\bf a},n, k} \stackrel{q}{\rightarrow} C_{{\bf a'},n-k-1, k-1}$.  

If $q(V_{\bullet}) = V_{\bullet}'$ then for each $j$ there exists some $m$ with $V_j \subseteq X^{-1} V_m'$. Since $V'_{\bullet} = (V'_1, \ldots, V'_{n-2}) = (V_1, \ldots, V_{i-1}, XV_{i+2}, \ldots, XV_{n})$ we may choose $m = j$ for $j\leq i$, $m = i$ for $i< j< i+3$, and $m = i-2$ for $j>i+2$.  

 Since $V'_{\bullet} \in C_{{\bf a'}, n-k-1, k-1}$ each $V_m' \subseteq V_{n-k-1,k-1}$ and so each $V_j \subseteq X^{-1}V_{n-k-1,k-1} = V_{n-k,k}$.  Thus for all noncrossing matchings ${\bf a}\in B^{n-k,k}$, each $C_{{\bf a},n-k,k}$ is contained in $Y_{n-k,k}$.  Since $\widetilde{\Phi}^{-1} \circ \ell$ is a diffeomorphism all $C_{{\bf a},n-k,k}$ are compact irreducible subvarieties of $Y_{n-k,k}$ with the same dimension as $X'_{n-k,k}$. Thus each $C_{{\bf a},n-k,k}$ is an irreducible component of the Springer variety $Y_{n-k,k}$.  Since $S'_{{\bf a},n-k,k} \neq S'_{{\bf b},n-k,k}$ for ${\bf a} \neq {\bf b}$ and $\widetilde{\Phi} \circ \ell$ is a diffeomorphism, it follows that $C_{{\bf a},n-k,k} \neq C_{{\bf b},n-k,k}$ for ${\bf a}\neq {\bf b}$.  

Recall that the irreducible components of  $\mathcal{S}_{n-k,k}$ are in bijection with standard Young tableaux of shape $(n-k,k)$. \cite[Proposition 1.3]{SW} shows the set of noncrossing matchings of type $(n-k,k)$ are in bijection with standard Young tableaux of shape $(n-k,k)$. These results together show that the irreducible components of $\mathcal{S}_{n-k,k}$ are indexed by noncrossing matchings of type $(n-k,k)$. Thus $\bigcup_{{\bf a} \in B^{n-k,k}} C_{{\bf a},n-k,k} \subseteq Y_{n-k,k}$ is isomorphic to the $(n, n-k)$ Springer variety, and we conclude $\bigcup_{{\bf a} \in B^{n-k,k}} C_{{\bf a},n-k,k} = Y_{n-k,k}$.
\end{proof}

\begin{theorem}\label{homeopf}
The $(n,n-k)$ Springer variety $\mathcal{S}_{n-k,k}$ is diffeomorphic to $X_{n-k,k}$.
\end{theorem}
\begin{proof}
Lemmas \ref{erase} and \ref{nk} imply that $\mathcal{S}_{n-k,k}$ is diffeomorphic to $X_{n-k,k}'$. Now we use an antipodal map to show that $X_{n-k,k}'$ and $X_{n-k,k}$ are diffeomorphic.

Define $\gamma: (S^2)^n \rightarrow (S^2)^n$ as $\gamma((x_1, \ldots , x_n)) = (-x_1, x_2, \ldots, (-1)^n x_n)$. This map is its own inverse and thus is a diffeomorphism. Since ${\bf a}\in B^{n-k,k}$ is noncrossing each arc in ${\bf a}$ has some number of arcs and no rays between its endpoints. This means that there are an even number of vertices between the endpoints of each arc and every arc $(i,j)\in {\bf a}$ has one even and one odd endpoint. Therefore $\gamma$ is the identity on exactly one of the coordinates $x_i$, $x_j$ and the antipodal map on the other. Furthermore given a ray $(i)\in {\bf a}$ the map $\gamma$ is the identity on $x_i$ if $i$ is even and the antipodal map if $i$ is odd. Hence $\gamma(X_{n-k,k}) = X_{n-k,k}'$. 
\end{proof}

As a concrete example we use our new construction to get a topological picture of the $(3,1)$ Springer variety. The $(3,1)$ Springer variety $X_{3,1}$ has three components indexed by the three elements of $B^{3,1}$ shown in Figure \ref{31match}.
\begin{figure}[h]
${\bf a} =$ \begin{picture}(60,20)(-10,0)
\qbezier(16, 0)(16,10)(8,10)
\qbezier(8,10)(0,10)(0,0)
\put(24,0){\line(0,1){20}}
\put(32,0){\line(0,1){20}}
\end{picture}
${\bf b} = $\begin{picture}(60,20)(-10,0)
\qbezier(24, 0)(24,10)(16,10)
\qbezier(16,10)(8,10)(8,0)
\put(0,0){\line(0,1){20}}
\put(32,0){\line(0,1){20}}
\end{picture}
${\bf c}  = $\begin{picture}(60,20)(-10,0)
\qbezier(16, 0)(16,10)(24,10)
\qbezier(24,10)(32,10)(32,0)
\put(0,0){\line(0,1){20}}
\put(8,0){\line(0,1){20}}
\end{picture}
\caption{Elements of $B^{3,1}$}\label{31match}
\end{figure}
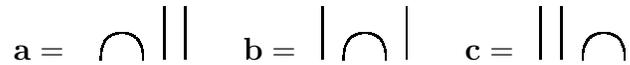

\noindent Each component of $X_{3,1}$ is homeomorphic to $S^2$.
\begin{itemize}
\item{$ S_{\bf a} = \{ (x,x,-p,p): x\in S^2\}$}
\item{$S_{\bf b} = \{ (-p,x,x,p): x\in S^2\}$}
\item{$S_{\bf c} = \{ (-p,p,x,x): x\in S^2\}$}
\end{itemize}
The components intersect in two different points.
\begin{itemize}
\item{$ S_{\bf a}\cap S_{\bf b} = \{ (-p,-p,-p,p)\} \in (S^2)^4\}$}
\item{$S_{\bf b}\cap S_{\bf c} = \{ (-p,p,p,p)\} \in (S^2)^4\}$}
\item{$S_{\bf a}\cap S_{\bf c} = \emptyset$}
\end{itemize}
Putting this information together we see that $X_{3,1}$ is the wedge of 3 spheres as shown in Figure \ref{31wedge}.

\begin{figure}[h]
\includegraphics[width=2in]{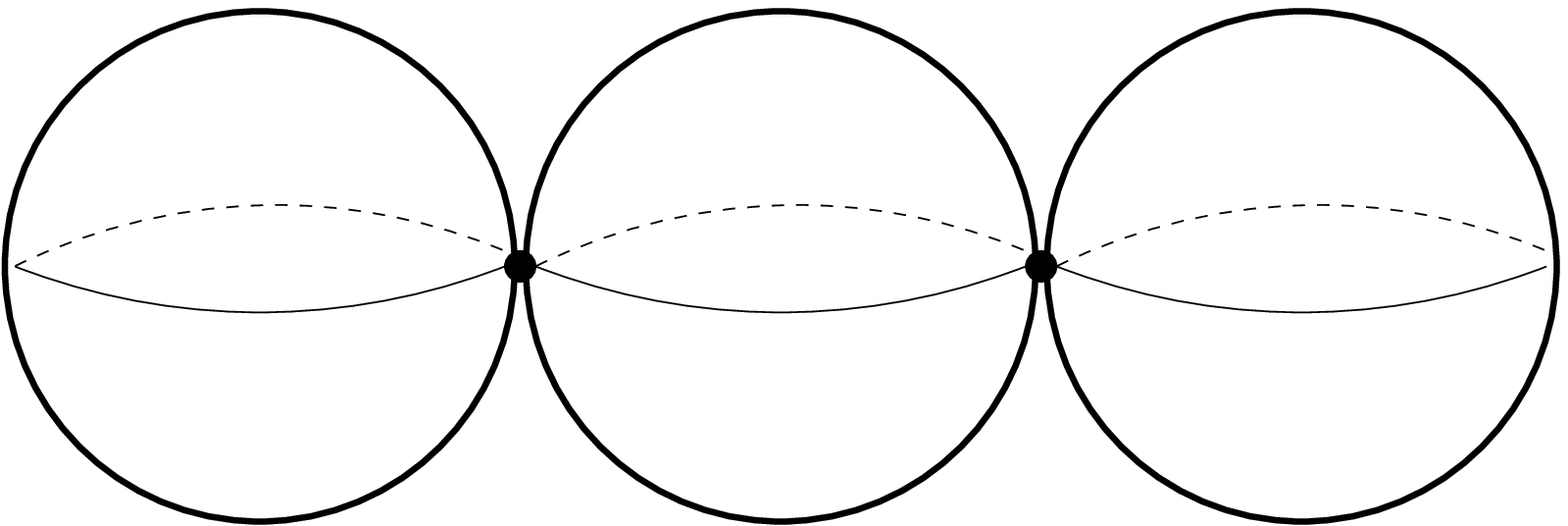}

\begin{picture}(100,10)(-50,0)
\put(-5,0){$S_{\bf b}$}
\put(-55,0){$S_{\bf a}$}
\put(45,0){$S_{\bf c}$}
\end{picture}
\caption{$X^{3,1}$ is a wedge of three copies of $S^2$.} \label{31wedge}
\end{figure}

\section{Intersections of components}\label{int}
From now on we regard $X_{n-k,k}$ as the $(n-k,k)$ Springer variety.  For fixed $n$ and $k$ we refer to $S_{{\bf a}, n-k,k}$ simply as $S_{\bf a}$. Following work of Khovanov \cite[Section 3]{K}  we prove a sequence of lemmas about intersection between components and use this information to set up a Mayer-Vietoris type exact sequence on homology. Propositions \ref{CircleDist} and \ref{Khovarrows} are stated without proof in \cite{K}. Because we need these results for our arguments, we include proofs in the Appendix to this paper. 

\subsection{Matchings and their associated components}

\begin{definition}
Given ${\bf a} ,{\bf b} \in B^{n-k,k}$ let ${\bf a}w({\bf b})$ be the result of reflecting ${\bf b}$ horizontally and gluing this reflection to ${\bf a}$. The one manifold ${\bf a}w({\bf b})$ will consist of circles, lines with both endpoints pointing up, lines with both endpoints pointing down, and lines with one endpoint pointing in each direction. See Figure \ref{revmatches} for an example. We define $|{\bf a}w({\bf b})|$ to be the number of connected components in ${\bf a}w({\bf b})$.  
\end{definition}

\begin{figure}[h]
\scalebox{.8}{\begin{picture}(200,100)(50,-80)
\put(-30,5){$a=$}
\put(0,0){\line(0,1){20}}
\qbezier(20, 0)(20,15)(30,15)
\qbezier(30,15)(40,15)(40,0)
\qbezier(60,0)(60,15)(70,15)
\qbezier(70,15)(80,15)(80,0)
\put(200,5){$b=$}
\put(310,0){\line(0,1){20}}
\qbezier(230, 0)(230,15)(240,15)
\qbezier(240,15)(250,15)(250,0)
\qbezier(270,0)(270,15)(280,15)
\qbezier(280,15)(290,15)(290,0)
\put(70,-45){$aw(b) =$}
\put(130,-50){\line(0,1){20}}
\put(210,-70){\line(0,1){20}}
\qbezier(150, -50)(150,-35)(160,-35)
\qbezier(160,-35)(170,-35)(170,-50)
\qbezier(190,-50)(190,-35)(200,-35)
\qbezier(200,-35)(210,-35)(210,-50)
\qbezier(130,-50)(130,-65)(140,-65)
\qbezier(140,-65)(150,-65)(150,-50)
\qbezier(170,-50)(170,-65)(180,-65)
\qbezier(180,-65)(190,-65)(190,-50)
\end{picture}}
\caption{Obtaining ${\bf a}w({\bf b})$ from ${\bf a}$ and ${\bf b}$. }\label{revmatches}
\end{figure}

\begin{definition}[Order on Noncrossing Matchings]
Given ${\bf a},{\bf b}\in B^{n-k,k}$ write ${\bf a}\rightarrow {\bf b}$ if one of the following is true
\begin{itemize}
\item{There is a quadruple $i<j<k<l$ where 
\begin{itemize}
\item[$\bullet$]{${\bf a}$ and ${\bf b}$ are identical off if $i,j,k,l$,}
\item[$\bullet$]{$(i,j), (k,l)\in {\bf a}$ and}
\item[$\bullet$]{$(i,l), (j,k)\in {\bf b}$.}
\end{itemize}}
\item{There is a triple $i<j<k$ where 
\begin{itemize}
\item[$\bullet$]{${\bf a}$ and ${\bf b}$ are identical off of $i,j,k$,}
\item[$\bullet$]{$(i), (j,k)\in {\bf a}$ and}
\item[$\bullet$]{ $(i,j), (k)\in {\bf b}$.}
\end{itemize}}
 \end{itemize}

Define a partial order ${\bf a}\prec {\bf b}$ if there exist a chain of arrows ${\bf a}\rightarrow \cdots \rightarrow {\bf b}$. Extend this to a total order $<$. Note that the extension to a total order is not unique. Figure \ref{order} has an example.
\end{definition}
 
 \begin{figure}[h]
 \includegraphics[width=5in]{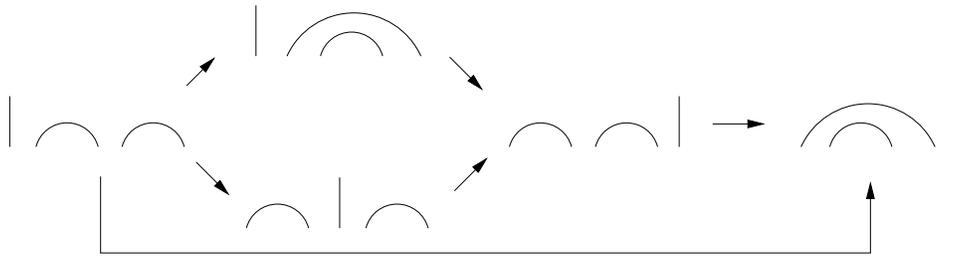}
 \caption{The $\rightarrow$ relation on elements of $B^{3,2}$.}\label{order}
 \end{figure}

\begin{definition}
Given ${\bf a},{\bf b}\in B^{n-k,k}$ let the distance from ${\bf a}$ to ${\bf b}$, denoted $d({\bf a},{\bf b})$,  be the minimal length $m$ of a sequence $({\bf a}={\bf a_0}, {\bf a_1}, \ldots, {\bf a_{m-1}}, {\bf a_m}={\bf b})$ such that ${\bf a_i} \rightarrow {\bf a_{i+1}}$ or ${\bf a_{i}} \leftarrow {\bf a_{i+1}}$ for all $i$. We call a sequence that realizes the distance between ${\bf a}$ and ${\bf b}$ a minimal sequence.
\end{definition} 

\begin{definition}
Given ${\bf a}\in B^{n-k,k}$ define the subset  $B_{\bf a}^{n-k,k}$ of $B^{n-k,k}$ as $$B_{\bf a}^{n-k,k} = \{ {\bf b}\in B^{n-k,k} : \textup{each line in } {\bf a}w({\bf b}) \textup{ has endpoints pointing in opposite directions} \}.$$
\end{definition}

The following observation of Fung proven in \cite[Theorem 7.3]{F} shows that the line segments in the one manifold ${\bf a}w({\bf b})$  tell us when two components have empty intersection.
\begin{proposition}[{\bf Fung}]
Given ${\bf a},{\bf b}\in B^{n-k,k}$ if ${\bf a}w({\bf b})$ has a line with both endpoints pointing in the same direction then $S_{\bf a}\cap S_{\bf b} = \emptyset$.
\end{proposition}
We will avoid these ``bad pairs" of matchings because they are unnecessary for the purpose of our arguments as illustrated by the following corollary. 
\begin{corollary}
Given ${\bf a}\in B^{n-k,k}$  $$\left( \bigcup_{\stackrel{{\bf b}<{\bf a}}{{\bf b}\in B^{n-k,k}}} S_{\bf b} \right) \bigcap S_{\bf a} = \left( \bigcup_{\stackrel{{\bf b} <{\bf a};}{ {\bf b}\in B_{\bf a}^{n-k,k}}} S_{\bf b} \right) \bigcap S_{\bf a}$$.
\end{corollary}
\noindent In the case that ${\bf b}\in B_{\bf a}^{n-k,k}$ the intersection $S_{\bf a} \cap S_{\bf b}$ is homeomorphic to $(S^2)^{\#\textup{ of circles in }{\bf a}w({\bf b})}$.

\begin{lemma}\label{raypair}
Let ${\bf a}\in B^{n-k,k}$ and ${\bf b}\in B_{\bf a}^{n-k,k}$. If we enumerate the rays in ${\bf a}$ and ${\bf b}$ from left to right as $i_i, \ldots, i_{n-2k}$ and $j_i, \ldots, j_{2n-k}$ respectively then in ${\bf a}w({\bf b})$ each line segment will contain rays $i_t$ and $j_t$ for some $1\leq t \leq n-2k$.
\end{lemma}

\begin{proof}
Each line segment in ${\bf a}w({\bf b})$ consists of two rays connected by some number of arcs. By assumption one ray is in ${\bf a}$ and one ray is in ${\bf b}$. Say that rays $i_{t_1}$ and $j_{t_2}$ are part of the same segment in ${\bf a}w({\bf b})$. Then the rays $i_i, \ldots, i_{t_1-1}, j_1, \ldots, j_{t_2-1}$ lie on one side of this line segment. Since ${\bf a}w({\bf b})$ is noncrossing these rays must be connected pairwise in a noncrossing manner. Since ${\bf b}\in B_{\bf a}^{n-k,k}$ each $i$-ray must be connected to a $j$-ray. Therefore $t_1 = t_2$.
\end{proof}

\begin{lemma}\label{break}
Let ${\bf a}\in B^{n-k,k}$ and ${\bf b}\in B_{\bf a}^{n-k,k}$, and consider a sequence $$({\bf a}={\bf a_0}, {\bf a_1}, \ldots, {\bf a_{m-1}},{\bf  a_m} ={\bf  b}).$$ If for all $i$, we have $|{\bf a_i}w({\bf b})| = |{\bf a_{i+1}}w({\bf b})| - 1$ then the sequence is minimal.
\end{lemma}

\begin{proof}
 For each $0\leq i \leq m-1$ either ${\bf a_i} \rightarrow {\bf a_{i+1}}$ or ${\bf a_i} \leftarrow {\bf a_{i+1}}$. Thus either an arc and a ray are changing position, two unnested arcs are replaced by two nested arcs, or two nested arcs are replaced by unnested arcs.  Consider the associated sequence $$({\bf a}w({\bf b}) = {\bf a_0}w({\bf b}), {\bf a_1}w({\bf b}), \ldots, {\bf a_{m-1}}w({\bf b}), {\bf a_{m}}w({\bf b}) = {\bf b}w({\bf b})).$$ 

 Analyzing each possibility, one can check that $|{\bf a_i}w({\bf b})| = |{\bf a_{i+1}}w({\bf b})| \pm 1$. Moreover, given ${\bf a_{i}}w({\bf b})$ and ${\bf a_{i+1}}w({\bf b})$, if the arc-ray or arc-arc pair that change from ${\bf a_i}$ to ${\bf a_{i+1}}$ are part of the same connected component in ${\bf a_{i}}w({\bf b})$ then ${\bf a_{i+1}}w({\bf b})$ will have one more component because we are splitting one component into two. On the other hand, if they are part of different connected components in ${\bf a_i}w({\bf b})$ then ${\bf a_{i+1}}w({\bf b})$ will have one less component because we are joining two components.
 
The one manifold ${\bf b}w({\bf b})$ has the maximum number of components possible, so $|{\bf a}w({\bf b})|\leq |{|bf b}w({\bf b})|$. Since each arrow move changes the number of components by 1, no  sequence from ${\bf a}$ to ${\bf b}$ can be shorter than $|{\bf b}w({\bf b})| - |{\bf a}w({\bf b})|$. The sequence given in the statement of the lemma has $|{\bf a}w({\bf b})| = |{\bf b}w({\bf b})| - m$. Since no shorter sequence can exist the given sequence is minimal. 
\end{proof}

Given ${\bf a},{\bf b}\in B^{n-k,k}$ label the $k$ connected components of ${\bf a}w({\bf b})$ with $c_1, \ldots, c_k$.  Associated to each $c_i$ let $$C_i = \{ 1\leq j \leq n : c_i \textup{ passes through vertex } j \textup{ in } {\bf a}w({\bf b})\}.$$ Let the collection of all of these sets be denoted by $\mathcal{C}_{{\bf a},{\bf b}} = \{ C_i\}$. A useful way to understand the distance between matchings is to study the sequence of collections $(\mathcal{C}_{{\bf a},{\bf b}}, \ldots, \mathcal{C}_{{\bf b},{\bf b}})$ associated to a sequence $({\bf a}={\bf a_0}, {\bf a_1},\ldots {\bf a_{m-1}}, {\bf a_m} ={\bf b})$. Restating Lemma \ref{break} in the language of the sets $\mathcal{C}_{{\bf a_i},{\bf b}}$ we have Lemma \ref{refine}.
\begin{lemma}\label{refine}
Let ${\bf a}\in B^{n-k,k}$ and ${\bf b}\in B_{\bf a}^{n-k,k}$, and consider a sequence $({\bf a} = {\bf a_0}, \ldots, {\bf a_m} = {\bf b})$. If for each $0\leq i \leq m-1$ the vertices of the arc-ray or arc-arc pair in ${\bf a_{i}}$ being changed when transitioning from $a_i$ to $a_{i+1}$ are part of a single set in $\mathcal{C}_{{\bf a_i},{\bf b}}$ then the sequence is minimal.
\end{lemma}

For $n$ even and ${\bf a}, {\bf b}\in B^{n/2, n/2}$ Khovanov asserts that a sequence of the type described in Lemmas \ref{break} and \ref{refine} always exists and that it gives rise to a nice geometric formula for the distance between two matchings \cite[pg 9]{K}.
\begin{proposition}[{\bf Khovanov}]\label{CircleDist}
For $n$ even and ${\bf a},{\bf b}\in B^{n/2, n/2}$ the distance $d({\bf a},{\bf b})$ is equal to $n/2 - |{\bf a}w({\bf b})|$.
\end{proposition}

Corollary \ref{distance} extends Khovanov's result to the case where ${\bf a}\in B^{n-k,k}$ and ${\bf b}\in B_{\bf a}^{n-k,k}$. In order to prove this we appeal to the notion of completing matchings of type $(n-k,k)$ to matchings of type $(n-k,n-k)$. This gives a natural way to see the set $B^{n-k,k}$ sitting inside of $B^{n-k,n-k}$.

\begin{definition}[Completion]
Given ${\bf a}\in B^{n-k,k}$ enumerate the rays in ${\bf a}$ from left to right by $(i_1), (i_2), \ldots, (i_{n-2k})$. Define the function $$\varphi : B^{n-k,k} \rightarrow B^{n-k,n-k}$$ where $\varphi({\bf a})$ is the matching with arcs
\begin{itemize}
\item{$(i+n-2k, j+n-2k)$ if $(i,j)$ is an arc in ${\bf a}$}
\item{$(n-2k-t+1, i_{t}+n-2k)$ if $(i_t)$ is a ray in ${\bf a}$} 
\end{itemize}
In other words, we complete each ray in ${\bf a}$ to an arc by adding additional vertices to the left of the matching and connecting unanchored endpoints of the rays to these vertices in the only way that keeps the matching noncrossing. 
We call $\varphi({\bf a})$ the \it{completion of {\bf a}}. Figure \ref{completion} has an example.
\end{definition}

Let $\widetilde{B}^{n-k,n-k}$ be the subset of $B^{n-k,n-k}$ consisting of matchings with no arcs of the form $(i,j)$ for $0\leq i<j\leq n-2k$. There is a natural inverse operation to completion on this subset of $B^{n-k,n-k}$.

\begin{definition}[Restriction]
Consider ${\bf a}\in \widetilde{B}^{n-k,n-k}$ Define the function $$\psi: \widetilde{B}^{n-k,n-k} \rightarrow B^{n-k,k}$$ where $\psi({\bf a})$ is the matching with 
\begin{itemize}
\item{arcs $(i-n+2k, j-n+2k)$ if $(i,j)$ is an arc in ${\bf a}$ with $n-2k<i$}
\item{rays $(j-n+2k)$ if $(i,j)$ is an arc in ${\bf a}$ with $1\leq i \leq n-2k$}
\end{itemize}
In other words, we remove the first $n-2k$ vertices of ${\bf a}$ so that each arc with left endpoint incident on one of those vertices becomes a ray.
We call $\psi({\bf a})$ the {\it restriction of ${\bf a}$}. Figure \ref{completion} has an example.
\end{definition}

\begin{figure}[h]
\begin{picture}(150,20)(-20,0)
\qbezier(10, 0)(10,10)(20,10)
\qbezier(20,10)(30,10)(30,0)
\put(50,0){\line(0,1){20}}
\qbezier(70,0)(70,10)(80,10)
\qbezier(80,10)(90,10)(90,0)
\put(110,0){\line(0,1){20}}
\put(130,15){$\stackrel{\varphi}{\longrightarrow}$}
\put(130,0){$\stackrel{\psi}{\longleftarrow}$}
\end{picture}
\begin{picture}(200,20)(-70,0)
\qbezier(-30,0)(-30,30)(40,30)
\qbezier(40,30)(110,30)(110,0)
\qbezier(10, 0)(10,10)(20,10)
\qbezier(20,10)(30,10)(30,0)
\qbezier(-10,0)(-10,20)(20,20)
\qbezier(20,20)(50,20)(50,0)
\qbezier(70,0)(70,10)(80,10)
\qbezier(80,10)(90,10)(90,0)

\end{picture}
\caption{A matching $a\in B^{4,2}$ and its completion $\varphi(a)\in B^{4,4}$}\label{completion}
\end{figure}
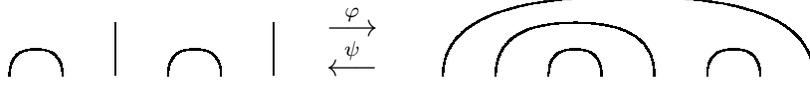

An immediate consequence of the definitions of restriction and completion is that $\varphi$ gives a one-to-one correspondence between the sets $B^{n-k,k}$ and $\widetilde{B}^{n-k,n-k}$. Note also that when $n$ is even and $k = n/2$ the sets $B^{n-k,k}$, $B^{n-k,n-k}$, and $\widetilde{B}^{n-k,n-k}$ are the same. Furthermore the maps $\varphi$ and $\psi$ are both the identity in this case.

\begin{proposition} \label{sequences}
Given ${\bf a}\in B^{n-k,k}$ and ${\bf b}\in B_{\bf a}^{n-k,k}$ every minimal sequence $(\varphi({\bf a}), \ldots, \varphi({\bf b}))$ gives a minimal sequence $({\bf a},\ldots, {\bf b})$.
\end{proposition}

\begin{proof}
Consider a minimal sequence $(\varphi({\bf a}) = {\bf x_0}, {\bf x_1}, \ldots, {\bf x_{m-1}}, {\bf x_m} = \varphi({\bf b}))$. Since this sequence is minimal Lemma \ref{refine} guarantees that for all $0\leq i< m$ the collection $\mathcal{C}_{{\bf x_i},\varphi({\bf b})}$ is a refinement of $\mathcal{C}_{{\bf x_{i-1}}, \varphi({\bf b})}$ in the sense that every set in $\mathcal{C}_{{\bf x_{i-1}}, \varphi({\bf b})}$ is a union of sets in $\mathcal{C}_{{\bf x_i},\varphi({\bf b})}$. In particular, $\mathcal{C}_{{\bf x_i},\varphi({\bf b})}$ contains exactly one more set than $\mathcal{C}_{{\bf x_{i-1}}, \varphi({\bf b})}$. 

By Lemma \ref{raypair} we know that the $i$th ray in ${\bf a}$ and the $i$th ray in ${\bf b}$ are part of the same line segment in ${\bf a}w({\bf b})$. Under $\varphi$ this line segment is completed to a circle passing through the newly added vertex $n-2k-i+1$. Thus each set in $\mathcal{C}_{\varphi({\bf a}), \varphi({\bf b})}$ contains at most one of the vertices $1, \ldots, n-2k$.  Indeed the sets coming from circles in $\varphi({\bf a})w(\varphi({\bf b}))$ that were line segments in ${\bf a}w({\bf b})$ each contain one of these vertices. 

Each set in $\mathcal{C}_{{\bf x_i},\varphi({\bf b})}$ is a refinement of the collection $\mathcal{C}_{\varphi({\bf a}),\varphi({\bf b})}$, so the sets in $\mathcal{C}_{{\bf x_i},\varphi({\bf b})}$ also each contain at most one of the vertices $1, \ldots, n-2k$. This means that for all $0\leq i \leq m$ the arcs in the matching ${\bf x_i}$ have at most one vertex incident on the vertices $1,\ldots, n-2k$. Therefore ${\bf x_i} \in \widetilde{B}^{n-k,n-k}$.  We can further conclude that whenever ${\bf x_i}\rightarrow {\bf x_{i+1}}$ then $\psi({\bf x_i})\rightarrow \psi({\bf x_{i+1}})$. Similarly whenever ${\bf x_i}\leftarrow {\bf x_{i+1}}$ then $\psi({\bf x_i})\leftarrow \psi({\bf x_{i+1}})$.

We conclude that the sequence $(\varphi({\bf a}) = {\bf x_0}, \ldots, {\bf x_m}=\varphi({\bf b}))$ has a corresponding sequence $({\bf a}, \psi({\bf x_1}), \ldots, \psi({\bf x_{m-1}}), {\bf b})$. 
Finally since
$$d(\varphi({\bf a}),\varphi({\bf b})) = n-k - |\varphi({\bf a})w(\varphi({\bf b}))| = n-k-|{\bf a}w({\bf b})|$$ we see that this sequence is minimal.
\end{proof}

Proposition \ref{sequences} allows us to extend Khovanov's geometric formula for the distance between matchings of type $(n/2, n/2)$ to matchings of type $(n-k,k)$.
\begin{corollary}\label{distance}
Let ${\bf a}\in B^{n-k,k}$ and ${\bf b}\in B_{\bf a}^{n-k,k}$. Then $d({\bf a},{\bf b}) = n-k-|{\bf a}w({\bf b})|$.
\end{corollary}

In \cite[Lemma 1]{K} Khovanov states that for $n$ even and ${\bf a}, {\bf b} \in B^{n/2, n/2}$ a very specific minimal sequence between ${\bf a}$ and ${\bf b}$ always exists. 

\begin{proposition}[{\bf Khovanov}] \label{Khovarrows}
Given ${\bf a}, {\bf b}\in B^{n/2, n/2}$ there exists ${\bf c}\in B^{n/2, n/2}$ with $d({\bf a},{\bf b}) = d({\bf a},{\bf c}) + d({\bf c},{\bf b})$ and ${\bf a}\succ {\bf c} \prec {\bf b}$.
\end{proposition}

\begin{lemma}\label{arrows}
Given ${\bf a}\in B^{n-k,k}$ and ${\bf b}\in B_{\bf a}^{n-k,k}$ there exists ${\bf c}\in B_{\bf a}^{n-k,k}$ with $d({\bf a},{\bf b}) = d({\bf a},{\bf c}) + d({\bf c},{\bf b})$ and ${\bf a}\succ {\bf c} \prec {\bf b}$.
\end{lemma}

\begin{proof}
Proposition \ref{Khovarrows} guarantees the existence of a minimal sequence 
$$\varphi({\bf a}) \leftarrow {\bf a_1} \leftarrow \cdots \leftarrow {\bf c'} \rightarrow \cdots \rightarrow {\bf a_{m-1}} \rightarrow \varphi({\bf b}).$$
The proof of Proposition \ref{sequences} implies that  for all $1\leq i\leq m$ we have that $\psi({\bf a_i})\in B_{\bf a}^{n-k,k}$ and $\psi({\bf a_i})\in B_{\bf b}^{n-k,k}$. This gives us the minimal sequence $${\bf a} \leftarrow \psi({\bf a_1}) \leftarrow \cdots \leftarrow \psi({\bf c'}) \rightarrow \cdots \rightarrow \psi({\bf a_{m-1}}) \rightarrow {\bf b}.$$ Setting $\psi({\bf c'})={\bf c}$ the lemma is proven.
\end{proof}

\begin{lemma}\label{intersect}
Let ${\bf a}\in B^{n-k,k}$ and ${\bf b}\in B_{\bf a}^{n-k,k}$. If $d({\bf a},{\bf c}) = d({\bf a},{\bf b}) + d({\bf b},{\bf c})$ then $$S_{\bf a}\cap S_{\bf c} = S_{\bf a}\cap S_{\bf b}\cap S_{\bf c}.$$
\end{lemma}

\begin{proof}
Clearly $S_{\bf a}\cap S_{\bf b}\cap S_{\bf b} \subseteq S_{\bf a}\cap S_{\bf c}$.

Since $d({\bf a},{\bf c}) = d({\bf a},{\bf b}) + d({\bf b},{\bf c})$ there exists a minimal sequence $({\bf a},\ldots, {\bf b}, \ldots, {\bf c})$. Consider the collections $\mathcal{C}_{{\bf a},{\bf c}} = \mathcal{C}_{{\bf c},{\bf a}} = \{C_1, \ldots, C_{i_{{\bf a},{\bf c}}}\}$, $\mathcal{C}_{{\bf a},{\bf b}} = \{ C'_1, \ldots, C'_{i_{{\bf a},{\bf b}}}\}$, and $\mathcal{C}_{{\bf b},{\bf c}} = \{ C''_1, \ldots, C''_{i_{{\bf b},{\bf c}}}\}$. The minimal sequence ensures that the collections $\mathcal{C}_{{\bf a},{\bf b}}$ and $\mathcal{C}_{{\bf b},{\bf c}}$ are each refinements of the collection $\mathcal{C}_{{\bf a},{\bf c}}$. 

In terms of sets in $\mathcal{C}_{{\bf a},{\bf c}}$ we have $$S_{\bf a}\cap S_{\bf c} = \{ x=(x_1, \ldots, x_n): x_i = x_j \textup{ if there is } C_k \in \mathcal{C}_{{\bf a},{\bf c}} \textup{ with } i,j\in C_k \}\subset S_{\bf a}.$$ Whenever $i,j\in C_k$ there are some $C'_{k_1}\in \mathcal{C}_{{\bf a},{\bf b}}$ and $C_{k_2}\in \mathcal{C}_{{\bf b},{\bf c}}$ with $C_k\subset C'_{k_1}$ amd $C_k\subset C''_{k_2}$. Then $x\in S_{\bf a}\cap S_{\bf b}$ and $x\in S_{\bf b}\cap S_{\bf c}$. Finally we conclude that $S_{\bf a}\cap S_{\bf c}\subseteq S_{\bf a}\cap S_{\bf b}\cap S_{\bf c}$.
\end{proof}

\begin{lemma}\label{intersection}
 Let $S_{<{\bf a}} \cap S_{\bf a} = (\cup_{{\bf b}<{\bf a}} S_{\bf b}) \cap S_{\bf a}$. Then $S_{<{\bf a}} \cap S_{\bf a} = \cup_{{\bf b}\rightarrow {\bf a}} (S_{\bf b}\cap S_{\bf a})$.
\end{lemma}

\begin{proof}
By basic set theory we have $S_{<{\bf a}} \cap S_{\bf a} = \cup_{{\bf b}<{\bf a}} (S_{\bf a}\cap S_{\bf b})$. Since ${\bf b}< {\bf a}$ whenever ${\bf b}\rightarrow {\bf a}$ we conclude that $\cup_{{\bf b}\rightarrow {\bf a}} (S_{\bf b}\cap S_{\bf a}) \subseteq S_{<{\bf a}} \cap S_{\bf a}$.  

Take $x\in S_{<{\bf a}}\cap S_{\bf a}$. Then there exists ${\bf b}<{\bf a}$ such that $x\in S_{\bf a}\cap S_{\bf b}$. By Lemma \ref{arrows} we have a minimal sequence $({\bf a}\leftarrow {\bf a_1} \cdots \leftarrow {\bf c} \rightarrow \cdots {\bf a_{m-1}} \rightarrow {\bf b})$. This sequence must begin with $\leftarrow$ since otherwise we would have a minimal sequence consisting entirely of $\rightarrow$ moves. The existence of such a sequence would mean that ${\bf a}<{\bf b}$, but we have assumed ${\bf b}<{\bf a}$. 

From the sequence we have $d({\bf a},{\bf b}) = d({\bf a}, {\bf a_1}) + d({\bf a_1},{\bf b})$. From Lemma \ref{intersect} we know 
$$S_{\bf a}\cap S_{\bf b} = S_{\bf a}\cap S_{{\bf a_1}} \cap S_{\bf b}\subseteq S_{\bf a}\cap S_{\bf a_1}.$$ We have shown that whenever ${\bf b}<{\bf a}$ there exists ${\bf a_1}\rightarrow {\bf a}$ with $S_{\bf a}\cap S_{\bf b} \subseteq S_{\bf a}\cap S_{{\bf a_1}}$. 
\end{proof}

\subsection{A useful cell decomposition}
Since each component $S_{\bf a}\in X_{n-k,k}$ is homeomorphic to a product of spheres, the components $S_{\bf a}$ have an obvious decomposition into even dimensional cells coming from their cartesian product structure. In Section \ref{hombas} we will explore the combinatorics of this decomposition in depth. It is useful in constructing an exact sequence on homology to appeal to another less obvious decomposition that generalizes the one given by Khovanov \cite[Lemma 4]{K}. 

\begin{definition}{(Decomposition of $S_a$)} \label{weirddecomp}
Given a matching ${\bf a}\in B^{n-k,k}$ we construct a graph $\Gamma = (V,E)$ with a vertex $v\in V$ for each arc in ${\bf a}$. Given $v\in V$ let $(v_1,v_2)$ with $v_1<v_2$ be the arc in $a$ corresponding to $v$. There is an edge between $v,w\in V$ if and only if there is some ${\bf b}\in B^{n-k,k}$ with ${\bf b}\rightarrow {\bf a}$ where ${\bf b}$ is identical to ${\bf a}$ off of vertices $v_1,v_2,w_1,w_2$. In ${\bf b}$ these vertices are connected in the only other possible noncrossing manner. 

The graph $\Gamma$ is a forest where each tree has a root corresponding to the outermost arc (i.e. the arc beneath which all other arcs with vertices in that tree are nested). Let $M$ be the set of all root vertices.  Figure \ref{graph} has an example.

Let $p = (0,0,1) \in S^2$ once again be the north pole of the standard embedded unit sphere. We construct a collection of cells $c(J)$ where $J \subseteq (E \cup M)$. 

Let $c(J) \subset \{ x = (x_1, \ldots, x_n)\in S_a\}$ such that
\begin{itemize}
\item{$x_{v_1} = x_{w_1}$ if edge $(v,w)\in J$}
\item{$x_{v_1} \neq x_{w_1}$ if edge $(v,w)\in E-J$}
\item{$x_{v_1} = (-1)^{v_1}p$ if vertex $v\in J$}
\item{$x_{v_1} \neq (-1)^{v_1}p$ if vertex $v\in M-J$}
\end{itemize}
Note that we only specify values for $x_{v_1}$ and $x_{w_1}$ because $x_{v_1} = x_{v_2}$ and $x_{w_1} = x_{w_2}$.
\end{definition} 

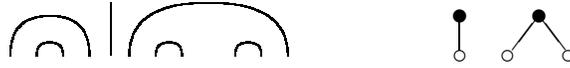
\begin{figure}[h]
\begin{picture}(200,20)(0,0)
\qbezier(10, 0)(10,5)(5,5)
\qbezier(5,5)(0,5)(0,0)
\qbezier(20,0)(20,15)(5,15)
\qbezier(5,15)(-10,15)(-10,0)
\put(28,0){\line(0,1){20}}
\qbezier(35,0)(35,20)(65,20)
\qbezier(65,20)(95,20)(95,0)
\qbezier(55, 0)(55,5)(50,5)
\qbezier(50,5)(45,5)(45,0)
\qbezier(85, 0)(85,5)(80,5)
\qbezier(80,5)(75,5)(75,0)
\put(160,15){\circle*{5}}
\put(160,2){\line(0,1){15}}
\put(160,0){\circle{4}}
\put(190,15){\circle*{5}}
\put(180,2){\line(3,4){10}}
\put(178,0){\circle{4}}
\put(200,2){\line(-3,4){10}}
\put(201,0){\circle{4}}
\end{picture}

\caption{A matching and its graph $\Gamma$.} \label{graph}
\end{figure}

\begin{lemma}
Each $c(J)$ is homeomorphic to $\mathbb{R}^{2k-|J|}$ and $S_a = \sqcup _{J} c(J)$.
\end{lemma}
\begin{proof}
This follows from the fact that these cells have a trivial vector bundle structure. The fact that we actually have a cell decomposition is clear by construction. 
\end{proof}

\begin{lemma} \label{subcomplex}
Given ${\bf b}\rightarrow {\bf a}$ the cell decomposition for $S_{\bf a}$ restricts to a cell decomposition for $S_{\bf a}\cap S_{\bf b}$. Therefore the decomposition for $S_{\bf a}$ also restricts to a decomposition for $S_{<{\bf a}}\cap S_{\bf a}$.
\end{lemma}

\begin{proof}
Consider ${\bf b}\rightarrow {\bf a}$. Then either
\begin{itemize}
\item{there exist $i<j<k<l$ with $(i,j),(k,l)\in {\bf b}$ and $(i,l), (j,k)\in {\bf a}$, or}
\item{there exist $i<j<k$ with $(i),(j,k)\in {\bf b}$ and $(i,j),(k)\in {\bf a}$.}
\end{itemize}
The matchings ${\bf a}$ and ${\bf b}$ are identical otherwise.

In the first case, no rays change position as we move from ${\bf b}$ to ${\bf a}$. Thus we get that $S_{\bf a}\cap S_{\bf b}$ is a subcomplex of $S_{\bf a}$ by Khovanov \cite[Lemma 4]{K}. 

Consider the second case. In order for both ${\bf a}$ and ${\bf b}$ to be noncrossing, there can be no arcs $(i',k')$ or rays $(i')$ where $i<i'<j<k<k'$. Every arc with left endpoint $i'$ with $i<i'<j$ must have right endpoint $k'$ with $k'<j$, so all arcs incident on vertices between $i$ and $j$ are completely contained between $i$ and $j$. Therefore there is an even number of vertices between $i$ and $j$, and so $i$ and $j$ have opposite parity. Since there is an even number of vertices between $j$ and $k$ the vertices $i$ and $k$ have the same parity. 

Let $\Gamma_{\bf a} = (V_{\bf a},E_{\bf a})$ with distinguished vertex set $M_{\bf a}$ be the graph for ${\bf a}$. The vertex $v\in V_{\bf a}$ corresponding to the arc $(i,j)\in {\bf a}$ is a root, so $v\in M_{\bf a}$. Given $J\subset (V_{\bf a}\cup M_{\bf a})$ whenever $v\in J$ all points $x=(x_1, \ldots, x_n)\in c(J)$ have $ x_i =x_j = (-1)^i p$. For all $J$ and for all $x\in c(J)$ we also have that $x_k = (-1)^k p = (-1)^i p$ since $(k)$ is a ray in ${\bf a}$. Since $S_{\bf a}\cap S_{\bf b} = \{ x\in S_{\bf a}: x_i = x_j = x_k = (-1)^{i}p \}$ the cells $c(J)$ for which $v\in J$ form a cell complex for $S_{\bf a}\cap S_{\bf b}$. 

Appealing to Lemma \ref{intersection} we have that $S_{<{\bf a}}\cap S_{\bf a}$ is a subcomplex of the complex for $S_{\bf a}$.
\end{proof}

\subsection{Building an exact sequence}

The following two statements are consequences of Lemma \ref{intersection}.

\begin{lemma}\label{inj}
The inclusion $(S_{<{\bf a}} \cap S_{\bf a}) \hookrightarrow S_{\bf a}$ induces an injective homomorphism on homology $H_*(S_{<{\bf a}}\cap S_{\bf a}) \hookrightarrow H_*(S_{\bf a})$.
\end{lemma}

\begin{lemma}\label{surj}
The homomorphism $\bigoplus_{{\bf b}\rightarrow {\bf a}} H_*(S_{\bf b}\cap S_{\bf a}) \stackrel{f'}{\longrightarrow} H_*(S_{<{\bf a}}\cap S_{\bf a})$ induced by inclusion is surjective.
\end{lemma}

Now consider the Mayer-Vietoris Sequence:

$$\cdots \rightarrow H_m(S_{<{\bf a}}\cap S_{\bf a}) \stackrel{f}{\longrightarrow} H_m(S_{<{\bf a}}) \oplus H_m(S_{\bf a}) \stackrel{g}{\longrightarrow} H_m(S_{\leq {\bf a}})\rightarrow \cdots$$ 

\begin{lemma}\label{MV}
The space $S_{\leq {\bf a}}$ has nonzero homology only in even degrees. The Mayer-Vietoris sequence above reduces to short exact sequences of the form
$$0 \rightarrow H_{2m}(S_{<{\bf a}}\cap S_{\bf a}) \stackrel{f}{\longrightarrow} H_{2m}(S_{<{\bf a}}) \oplus H_{2m}(S_{\bf a}) \stackrel{g}{\longrightarrow} H_{2m}(S_{\leq {\bf a}}) \rightarrow 0.$$
\end{lemma}
\begin{proof}
We prove this inductively with respect to the order $<$. The base case is obvious since for the first ${\bf a}$ in  the order $S_{\leq {\bf a}} = S_{\bf a}$.  Assume the lemma is true for all matchings up to and including some matching ${\bf e}$, and let ${\bf a}$ be the next matching in the order. Then $S_{<{\bf a}} = S_{\leq {\bf e}}, S_{\bf a}$, and $S_{<{\bf a}}\cap S_{\bf a}$ all have homology in even degrees only.  Therefore the Mayer-Vietoris sequence decomposes into exact sequences 
$$0\rightarrow H_{2m+1}(S_{\leq {\bf a}}) \stackrel{\partial}{\longrightarrow} H_{2m}(S_{<{\bf a}}\cap S_{\bf a}) \stackrel{f}{\longrightarrow} H_{2m}(S_{<{\bf a}}) \oplus H_{2m}(S_{\bf a}) \stackrel{g}{\longrightarrow} H_{2m}(S_{\leq {\bf a}}) \rightarrow 0.$$

By Lemma \ref{inj} the map $H_{2m}(S_{<{\bf a}}\cap S_{\bf a}) \rightarrow H_{2m}(S_{\bf a})$ is injective so $H_{2m+1}(S_{\leq {\bf a}}) = 0$ as desired. 
\end{proof}

\begin{theorem}\label{sequence}
The sequence $$\bigoplus_{{\bf b}\rightarrow {\bf c}\leq {\bf a}} H_*(S_{\bf b}\cap S_{\bf c}) \stackrel{\psi^-}{\longrightarrow} \bigoplus_{{\bf b}\leq {\bf a}} H_*(S_{\bf b}) \stackrel{\phi}{\longrightarrow} H_*(S_{\leq {\bf a}}) \rightarrow 0$$ is exact. Here $\psi^-$ is defined to be $\Sigma_{{\bf b}\rightarrow {\bf c} \leq {\bf a}} (\psi_{{\bf b},{\bf c}} - \psi_{{\bf c},{\bf b}})$ where $\psi_{{\bf b},{\bf c}}: H_*(S_{\bf b}\cap S_{\bf c}) \rightarrow H_*(S_{\bf b})$ is induced by the inclusion $S_{\bf b}\cap S_{\bf c} \subset S_{\bf b}$ and $\psi_{{\bf c},{\bf b}}$ is similarly defined. The map $\phi$ is induced by the inclusions $S_{\bf b}\subset S_{\leq {\bf a}}$.
\end{theorem}

\begin{proof}
This proof follows the structure of the proof of \cite[Proposition 4]{K}. We prove this result by induction with respect to the total order on the set $B^{n-k,k}$. Let ${\bf a_0}$ be the first element with respect to this order. Then there do not exist ${\bf b},{\bf c}\leq {\bf a_0}$ and $S_{\leq {\bf a_0}} = S_{{\bf a_0}}$. The sequence
$$ 0\rightarrow H_*(S_{{\bf a_0}}) \stackrel{\phi}{\longrightarrow} H_*(S_{{\bf a_0}}) \rightarrow 0$$ is exact since $\phi$ is an isomorphism, and the base case is proven.

Assume for some element ${\bf a_{i-1}}\in B^{n-k,k}$ that the sequence
$$\bigoplus_{{\bf b}\rightarrow {\bf c}\leq {\bf a_{i-1}}} H_*(S_{\bf b}\cap S_{\bf c}) \stackrel{\psi^-}{\longrightarrow} \bigoplus_{{\bf b}\leq {\bf a_{i-1}}} H_*(S_{\bf b}) \stackrel{\phi}{\longrightarrow} H_*(S_{\leq {\bf a_{i-1}}}) \rightarrow 0$$ is exact. Consider the next element ${\bf a_i}$. By Lemma \ref{MV} we know that 
$$0 \rightarrow H_{*}(S_{<{{\bf a_i}}}\cap S_{{\bf a_i}}) \stackrel{f}{\longrightarrow} H_{*}(S_{<{\bf a_i}}) \oplus H_{*}(S_{{\bf a_i}}) \stackrel{g}{\longrightarrow} H_{*}(S_{\leq {\bf a_i}}) \rightarrow 0$$ is an exact sequence.  By Lemma \ref{surj} the map $f':\bigoplus_{{\bf b}\rightarrow {\bf c}\leq {\bf a_i}} H_*(S_{\bf b}\cap S_{\bf c}) \rightarrow H_*(S_{<{\bf a_i}}\cap S_{{\bf a_i}})$ is surjective so $Im(f) = Im(f\circ f')$. Thus we get the exact sequence
$$\bigoplus_{{\bf b}\rightarrow {\bf c}\leq {\bf a_i}} H_*(S_{\bf b}\cap S_{\bf c})\stackrel{f'\circ f}{\longrightarrow} H_{*}(S_{<{\bf a_i}}) \oplus H_{*}(S_{{\bf a_i}}) \stackrel{g}{\longrightarrow} H_{*}(S_{\leq {\bf a_i}}) \rightarrow 0.$$

Now consider the sequence 
$$\bigoplus_{{\bf b}\rightarrow {\bf c}\leq {{\bf a_i}}} H_*(S_{\bf b}\cap S_{\bf c}) \stackrel{\psi^-}{\longrightarrow} \bigoplus_{{\bf b}\leq {\bf a_i}} H_*(S_{\bf b}) \stackrel{\phi}{\longrightarrow} H_*(S_{\leq {\bf a_i}}) \rightarrow 0.$$ Since the inclusions $S_{\bf b}\cap S_{\bf c} \subset S_{<{\bf a_i}} \subset S_{\leq {\bf a_i}}$ and $S_{\bf b}\cap S_{\bf c} \subset S_{\bf b} \subset S_{\leq {\bf a_i}}$ and $S_{\bf b}\cap S_{\bf c} \subset S_{\bf c}\subset S_{\leq {\bf a_i}}$ all induce the same maps on homology we see that $Im(\psi^-) = Ker(\phi)$. Take some $y\in H_*(S_{\leq {\bf a_i}})$. By exactness there exists $(x,x')\in H_{*}(S_{<{\bf a_i}}) \oplus H_{*}(S_{{\bf a_i}})$ with $g(x,x') = y$. Again by exactness, there exists some $x''\in \bigoplus_{{\bf b}\leq {\bf a_{i-1}}} H_*(S_{\bf b})$ with $\phi(x'') = x$. Then we have that $\phi(x'',x') = g(x,x') = y$, and the desired sequence is exact.
\end{proof}

\section{Diagrammatic Homology Generators}\label{hombas}

Let ${\bf a},{\bf b}\in B^{n-k,k}$ with ${\bf a}\rightarrow {\bf b}$. Recall that $S_{\bf a}$ is diffeomorphic to $(S^2)^k$ and $S_{\bf a}\cap S_{\bf b}$ is diffeomorphic to $(S^2)^{k-1}$. In this section we discuss the cartesian product cell decomposition for $S_{\bf a}$ giving a convenient diagrammatic way to represent the cells.

Let $E$ be the set of arcs in ${\bf a}$, and for each $e\in E$ let $e_{\ell}$ be the left endpoint of the arc $e$ and $e_r$ be the right endpoint. Note that for each $x = (x_1, \ldots, x_n)\in S_{\bf a}$ we have that $x_{e_{\ell}} = x_{e_r}$. Therefore to describe a subset of $S_{\bf a}$ we only need to specify the value of $x_{e_{\ell}}$. 

Consider the cell decomposition of the two-sphere given by the point $p = (0,0,1)\in S^2$ and the two-cell $c = S^2-p$. 
\begin{definition} {(Diagrammatic Homology Basis for $S_{\bf a}$)}\label{diagdec}

Given $I\subset E$ we define the following cell $c(I)$ which is homeomorphic to $\mathbb{R}^{2|I|}$.
$$c(I) = \{ x = (x_1, \ldots, x_n)\in S_{\bf a} : x_{e_{\ell}} = p \textup{ if } e\notin I \textup{ and } x_{e_{\ell}} \neq p \textup{ if } e\in I \}.$$
Varying over all possible subsets of $E$ we get the cartesian product cell decomposition for $S_{\bf a}$.
Note that this decomposition is in general not the same as the one described in Definition \ref{weirddecomp}. 
We can diagrammatically represent the cartesian product homology generators by considering certain markings of the matching ${\bf a}$. 

Given $I\subset E$ let $M$ be the matching ${\bf a}$ decorated as follows:
\begin{itemize}
\item{Mark each ray with a dot.}
\item{For each $e\in I$ mark the corresponding arc in ${\bf a}$ with a dot.}
\item{Leave all other arcs unmarked.}
\end{itemize}

We call such a diagram a dotted noncrossing matching for ${\bf a}$. These will be denoted throughout by capital letters. Figure \ref{markedmatch} has an example.
\end{definition}

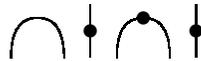
\begin{figure}[h]
\begin{picture}(150,20)(-20,0)
\qbezier(10, 0)(10,15)(20,15)
\qbezier(20,15)(30,15)(30,0)
\put(40,0){\line(0,1){20}}
\qbezier(50,0)(50,15)(60,15)
\qbezier(60,15)(70,15)(70,0)
\put(80,0){\line(0,1){20}}
\put(60,15){\circle*{5}}
\put(80,10){\circle*{5}}
\put(40,10){\circle*{5}}
\end{picture}
\caption{A dotted noncrossing matching $M$ for ${\bf a}$.}\label{markedmatch}
\end{figure}

We use a similar diagrammatic convention to represent cartesian product cells for $(S^2)^n$ in \cite[Definition 3.1]{RT}. These diagrams, which we call line diagrams, are collections of  $n$ dotted and undotted vertical lines. The common theme of dotted noncrossing matchings and line diagrams is that unmarked lines and arcs correspond to two-cells while dotted lines and arcs correspond to points. 
We want to consider the behavior of these generators under maps induced by inclusion.

\begin{definition}
Consider ${\bf a}\rightarrow {\bf b} \in B^{n-k,k}$. Then either ${\bf a}$ and ${\bf b}$ agree off of two arcs or an arc and a ray. The first two relations deal with the case when ${\bf a}$ and ${\bf b}$ differ on two arcs. The third relation deals with the case when they differ on an arc and a ray.
\begin{enumerate}
\item {\bf Type I relations:}  Let $M_1$ be the dotted noncrossing matching with dotted arc $(i,j)$  and undotted arc $(k,l)$.  Let $M_2$ be the dotted noncrossing matching with dotted arc $(k,l)$ and undotted arc $(i,j)$.  Define $M'_1$ and $M'_2$ analogously for ${\bf b}$, so that they agree with the $M_i$ off of $i,j,k,l$.  Type I relations have the form
\[M_1+M_2-M'_1-M'_2=0.\]
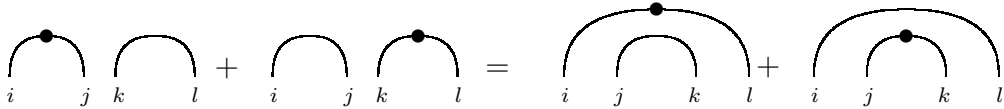
\begin{figure}[h]
\begin{equation*}
 \begin{picture}(85,10)(0,0)
\qbezier(10, 0)(10,15)(24,15)
\qbezier(24,15)(38,15)(38,0)
\qbezier(50,0)(50,15)(65,15)
\qbezier(65,15)(80,15)(80,0)
\put(24,15){\circle*{5}}
 \put(9,-10){\tiny{$i$}}
 \put(37,-10){\tiny{$j$}}
 \put(49,-10){\tiny{$k$}}
 \put(79,-10){\tiny{$l$}}
 \end{picture} + 
 \begin{picture} (80,10)(0,0)
 \qbezier(10, 0)(10,15)(24,15)
\qbezier(24,15)(38,15)(38,0)
\qbezier(50,0)(50,15)(65,15)
\qbezier(65,15)(80,15)(80,0)
\put(65,15){\circle*{5}}
 \put(9,-10){\tiny{$i$}}
 \put(37,-10){\tiny{$j$}}
 \put(49,-10){\tiny{$k$}}
 \put(79,-10){\tiny{$l$}}
 \end{picture} \hspace{.1in} = \hspace{.1in} 
 \begin{picture}(80,10)(0,0)
 \qbezier(10, 0)(10,25)(45,25)
\qbezier(45,25)(80,25)(80,0)
\qbezier(30,0)(30,15)(45,15)
\qbezier(45,15)(60,15)(60,0)
\put(45,25){\circle*{5}}
 \put(9,-10){\tiny{$i$}}
 \put(29,-10){\tiny{$j$}}
 \put(57,-10){\tiny{$k$}}
 \put(79,-10){\tiny{$l$}}
 \end{picture} + 
 \begin{picture}(80,10)(0,0)
 \qbezier(10, 0)(10,25)(45,25)
\qbezier(45,25)(80,25)(80,0)
\qbezier(30,0)(30,15)(45,15)
\qbezier(45,15)(60,15)(60,0)
\put(45,15){\circle*{5}}
 \put(9,-10){\tiny{$i$}}
 \put(29,-10){\tiny{$j$}}
 \put(57,-10){\tiny{$k$}}
 \put(79,-10){\tiny{$l$}}
 \end{picture}
\end{equation*}
\caption{Type I relation}\label{tp1}
\end{figure}
\item {\bf Type II relations:} Let $M_3$ be the dotted noncrossing matching with dotted arcs $(i,j)$ and $(k,l)$. Define $M'_3$ analogously for ${\bf b}$.  Type II relations have the form
\[M_3-M'_3=0.\]
\begin{figure}[h]
\begin{equation*}
 \begin{picture}(85,10)(0,0)
\qbezier(10, 0)(10,15)(24,15)
\qbezier(24,15)(38,15)(38,0)
\qbezier(50,0)(50,15)(65,15)
\qbezier(65,15)(80,15)(80,0)
\put(24,15){\circle*{5}}
\put(65,15){\circle*{5}}
 \put(9,-10){\tiny{$i$}}
 \put(37,-10){\tiny{$j$}}
 \put(49,-10){\tiny{$k$}}
 \put(79,-10){\tiny{$l$}}
 \end{picture}
 = 
 \begin{picture}(80,10)(0,0)
 \qbezier(10, 0)(10,25)(45,25)
\qbezier(45,25)(80,25)(80,0)
\qbezier(30,0)(30,15)(45,15)
\qbezier(45,15)(60,15)(60,0)
\put(45,25){\circle*{5}}
\put(45,15){\circle*{5}}
 \put(9,-10){\tiny{$i$}}
 \put(29,-10){\tiny{$j$}}
 \put(57,-10){\tiny{$k$}}
 \put(79,-10){\tiny{$l$}}
 \end{picture}
\end{equation*}
\caption{Type II relation}\label{tp2}
\end{figure}
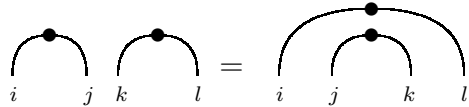
\item{\bf Type III relations:} Let $M_4$ be the dotted noncrossing matching with dotted ray $(i)$ and dotted arc $(j,k)$. Let $M'_4$ be the dotted noncrossing matching identical to $M_4$ except with dotted arc $(i,j)$ and dotted ray $(k)$. Type III relations have the form
\[M_4-M'_4=0 \]

\begin{figure}[h]
\begin{equation*}
\begin{picture}(50,20)(0,0)
\qbezier(0,0)(0,15)(15,15)
\qbezier(15,15)(30,15)(30,0)
\put(40,0){\line(0,1){24}}
\put(15,15){\circle*{5}}
\put(40,13){\circle*{5}}
\put(-2,-10){\tiny{$i$}}
\put(26,-10){\tiny{$j$}}
\put(37,-10){\tiny{$k$}}
\end{picture}
=
\begin{picture}(50,20)(-10,0)
\qbezier(10,0)(10,15)(25,15)
\qbezier(25,15)(40,15)(40,0)
\put(0,0){\line(0,1){24}}
\put(25,15){\circle*{5}}
\put(0,13){\circle*{5}}
\put(-2,-10){\tiny{$i$}}
\put(8,-10){\tiny{$j$}}
\put(38,-10){\tiny{$k$}}
\end{picture}
\end{equation*}
\end{figure}
\end{enumerate}
\end{definition}

\begin{theorem}
Let $\mathcal{T}$ be the submodule of $\bigoplus_{{\bf b}} H_*(S_{\bf b})$ generated by all Type I, Type II, and Type III relations. Then there is the following isomorphism. 
$$H_*(X_{n-k,k}) \cong \left( \bigoplus_{\bf b} H_*(S_{\bf b}) \right)/ \mathcal{T}$$
\end{theorem}
\begin{proof}
From Theorem \ref{sequence} we see that $\bigoplus_{{\bf b}} H_*(S_{\bf b})$ generates $H_*(X_{n-k,k})$ with relations coming from $\psi^-(H_*(S_{\bf a}\cap S_{\bf b}))$ where ${\bf a}\rightarrow {\bf b}$. The following calculations show that these relations always have the form of Type I, II, or III relations. We use the notation for cartesian product cells found in Definition \ref{diagdec}.

Consider ${\bf a},{\bf b}\in B^{2,2}$ where ${\bf a}$ has arcs $(1,2), (3,4)$ and ${\bf b}$ has arcs $(1,4), (2,3)$. Then ${\bf a}w({\bf b})$ consists of one circle, and $S_{\bf a}\cap S_{\bf b}$ is homeomorphic to $S^2$. Decompose $S_{\bf a} \cap S_{\bf b}$ into
\begin{itemize}
\item{$c_1 = \{ (p,p,p,p)\in S_{\bf a}\cap S_{\bf b}\} \cong \mathbb{R}^0$ and}
\item{$c_2 = \{ (x,x,x,x) \in S_{\bf a}\cap S_{\bf b}: x\neq p\} \cong \mathbb{R}^2$.}
\end{itemize} 
Under the inclusion map $\iota_{\bf a}: H_*(S_{\bf a}\cap S_{\bf b}) \hookrightarrow H_*(S_{\bf a})$ we have that $\iota_{\bf a}(c_1) = c(\emptyset)$ and $\iota_{\bf a}(c_2) = c(\{(1,2)\}) + c(\{(3,4)\})$. Similarly under the inclusion map $\iota_{\bf b}: H_*(S_{\bf a}\cap S_{\bf b}) \hookrightarrow H_*(S_{\bf b})$ we have that $\iota_{\bf b}(c_1) = c(\emptyset)$ and $\iota_{\bf b}(c_2) = c(\{(1,4)\}) + c(\{(2,3)\})$. This proves that $\psi^-(H_*(S_{\bf a}\cap S_{\bf b}))$ is generated by Type I and Type II relations whenever ${\bf a}$ and ${\bf b}$ differ on two arcs. 

Consider ${\bf a},{\bf b}\in B^{3,1}$ where ${\bf a}$ has ray $(1)$ and arc $(2,3)$ while ${\bf b}$ has ray $(3)$ and arc $(1,2)$. In this case ${\bf a}w({\bf b})$ consists of one line and no circles, so $S_{\bf a}\cap S_{\bf b}$ is just a point. Call this point $c_1$. Under the inclusion map $\iota_{\bf a}: H_*(S_{\bf a}\cap S_{\bf b}) \hookrightarrow H_*(S_{\bf a})$ we have that $\iota_{\bf a}(c_1) = c(\emptyset)$. Under the inclusion $\iota_{\bf b}: H_*(S_{\bf a}\cap S_{\bf b}) \hookrightarrow H_*(S_{\bf b})$ we get $\iota_{\bf b}(c_1) = c(\emptyset)$. This proves that whenever ${\bf a}$ and ${\bf b}$ differ on an arc and a ray, the image $\psi^-(H_*(S_{\bf a}\cap S_{\bf b}))$ is generated by Type III relations. 
\end{proof}

\begin{definition}
A standard dotted noncrossing matching is a dotted noncrossing matching where no dotted arc is nested beneath any other arc and no ray is to the right of any dotted arc. Figure \ref{tabmat} has an example.
\end{definition}

\begin{lemma}\label{gens}
The set of standard noncrossing matchings contains a set of generators for $H_*(X_{n-k,k})$
\end{lemma}
\begin{proof}
Type I and Type II relations allow us to reduce to a set of generators where no dotted arcs are nested beneath undotted or dotted arcs. Type III relations allow us to move rays to the left of all dotted arcs.
\end{proof}

Results of De Concini-Procesi \cite[p 213]{DP} and Garsia-Procesi \cite[Equation 4.2]{GP} specialized to the two-row case yield the following result which gives us the dimension of $H_*(X_{n-k,k})$ in each degree.

\begin{proposition}[{\bf De Concini-Procesi, Garsia-Procesi}] \label{GP}
For for each $m \leq k$ the cohomology $H^{2m}(X_{n-k,k})$ is the irreducible representation of the symmetric group $S_n$ corresponding to the partition $(n-m,m)$.
\end{proposition}

Young tableaux theory is often used to study irreducible representations of the symmetric group. Recall that a Young diagram corresponding to a partition of the number $n$ is a collection of top and left aligned boxes with row lengths given by the partition. A standard Young tableau is a filling of a Young diagram with the numbers $1$ through $n$ so that numbers strictly increase from left to right and top to bottom. For more information on Young tableaux and the Specht module theory we will use in the next section we refer the reader to \cite[Chapter 7]{Ful}. 

\begin{lemma}\label{biject}
Standard noncrossing matchings of type $(n-k,k)$ with $m$ undotted arcs are in bijection with standard Young tableaux of shape $(n-m,m)$.
\end{lemma}
\begin{proof}
We give an explicit bijection between standard dotted noncrossing matchings and standard Young tableaux.

Let $M$ be a standard dotted noncrossing matching. Define $g(M)$ to be the tableau obtained by writing the right hand endpoint of each undotted arc of $M$ on the bottom row and all other numbers on the top. The tableau $g(M)$ is standard since for each number in the bottom row, the left endpoint of that arc is written in the top row. The tableau has shape $(n-m,m)$ where $m$ is the number of undotted arcs in $M$.

Now let $T$ be a standard tableau of shape $(n-m,m)$. We define $h(T)$ to be the standard dotted noncrossing matching with $n-2k$ rays and $m$ undotted arcs constructed as follows. Start with the smallest number in the bottom row, draw an undotted arc with right endpoint incident on that number and left endpoint incident on the nearest unoccupied vertex. Repeat this for all remaining numbers in the bottom row.  Then in the leftmost $n-2k$ empty positions place $n-2k$ rays. Fill the remaining empty positions with unnested dotted arcs.  Figure \ref{tabmat} has an example of this map.

The proof that $g\circ h$ and $h \circ g$ produce the identity on the sets of standard tableaux and standard dotted noncrossing matchings respectively is a straightforward generalization of the one given in \cite[Section 2.2]{RT}.
\end{proof}

\begin{figure}[h]
$\begin{tabular}{|c|c|c|c|c|}
\cline{1-5} 1 & 2 & 4 & 5 & 7 \\
\cline{1-5} 3 & 6 & \multicolumn{2}{c}{} \\
\cline{1-2} \multicolumn{4}{c}{} \vspace{-1.3em} \end{tabular} \hspace{.5in} \stackrel{h}{\longrightarrow} \hspace{.5in}$
\begin{picture}(100,20)(0,5)
\put(0,0){\line(0,1){20}}
\qbezier(10,0)(10,10)(20,10)
\qbezier(20,10)(30,10)(30,0)
\qbezier(40,0)(40,20)(60,20)
\qbezier(60,20)(80,20)(80,0)
\qbezier(50,0)(50,10)(60,10)
\qbezier(60,10)(70,10)(70,0)
\put(0,10){\circle*{5}}
\put(60,20){\circle*{5}}
\end{picture}

\caption{A standard tableau of shape $(5,2)$ and its associated standard dotted noncrossing matching of type $(4,3)$.}\label{tabmat}
\end{figure}
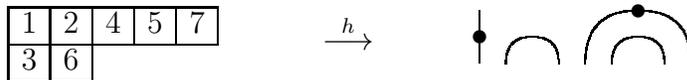

\begin{corollary}
The standard noncrossing matchings on $n$ vertices with $n-2k$ rays and $m$ undotted arcs form a basis for 
$H_{2m}(X_{n-k,k})$.
\end{corollary}
\begin{proof}
By Proposition \ref{GP} the dimension of $H_{2m}(X_{n-k,k})$ is the same as the dimension of the irreducible representation of $S_n$ corresponding to the partition $(n-m,m)$. Classical representation theory tells us that this dimension is equal to the number of standard Young tableaux of shape $(n-m,m)$ \cite{Ful}.

By Lemma \ref{gens} the set of standard dotted noncrossing matchings with $m$ undotted arcs is a generating set for $H_{2m}(X_{n-k,k})$. Lemma \ref{biject} tells us that this generating set has the proper dimension and therefore is a basis.
\end{proof}

\section{Springer action on $H_*(X_{n-k,k})$}\label{action}

Consider $n$ and $0\leq k\leq \lfloor n/2 \rfloor$. In a previous work we define a graded action of the symmetric group $S_{2n-2k}$ on $H_*(X_{n-k,n-k})$ \cite{RT}. We prove that this action which can be combinatorially and diagrammatically described is the Springer action. In other words $H_{2m}(X_{n-k,n-k})$ is the irreducible representation of $S_{2n-2k}$ corresponding to the partition $(2n-2k-m, m)$ for each $0\leq m\leq n-k$. This action is summarized by the chart in Figure \ref{snact}.

Using the same approach as \cite[Section 4]{RT} we construct the Springer action on all two-row Springer varieties. We once again describe this action diagrammatically. Finally we give a convenient skein-theoretic method for calculating the action.
\begin{figure}
\scalebox{.75}{
\begin{tabular}{|c||c|c|}
\cline{1-3} & &   \hspace{1em} \\
 {\bf Vertex Labelings} & $M$ & $(i~~~~ i+1)\cdot M$  \\
& &   \hspace{1em} \\
 \cline{1-3} & &  \hspace{1em} \\
 {\bf Case 1} &\includegraphics[width=.5in]{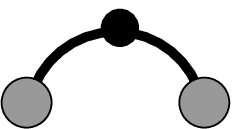}& \includegraphics[width=.5in]{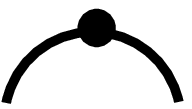}\\
 \cline{2-3} & & \hspace{1em} \\
 \raisebox{7pt}{\small{$i,i+1$ both on dotted arcs}}&  \includegraphics[width=1in]{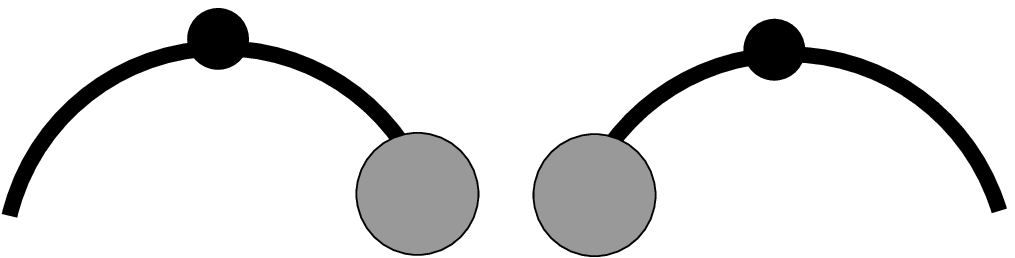} &  \includegraphics[width=1in]{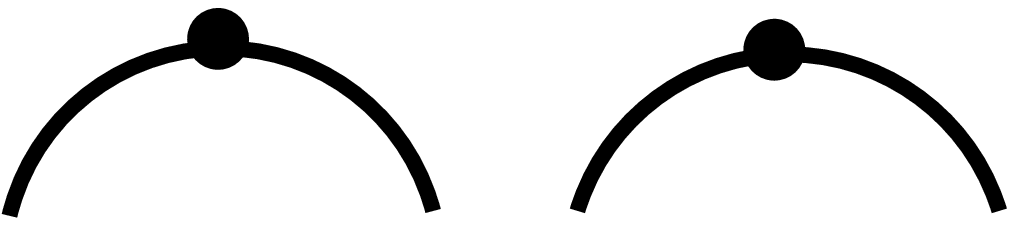} \\
 \cline{1-3} & & \hspace{1em}\\
 {\bf  Case 2} & \includegraphics[width=.5in]{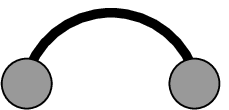} &  $-$ \includegraphics[width=.5in]{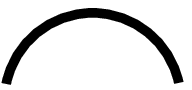} \\
 \raisebox{7pt}{\small{$(i,i+1)$ is an undotted arc}} & & \\
 \cline{1-3} & & \hspace{1em}\\
  {\bf  Case 3}  &\includegraphics[width=1in]{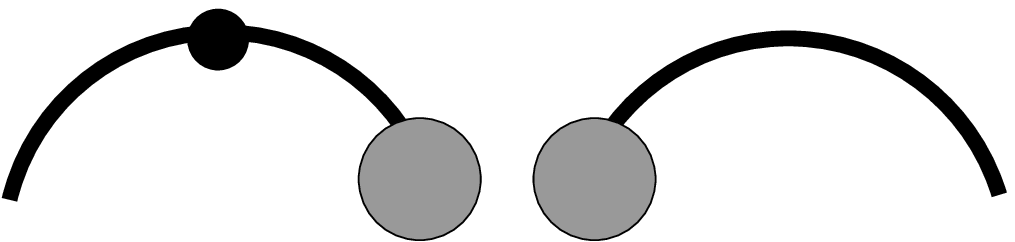} & $\includegraphics[width=1in]{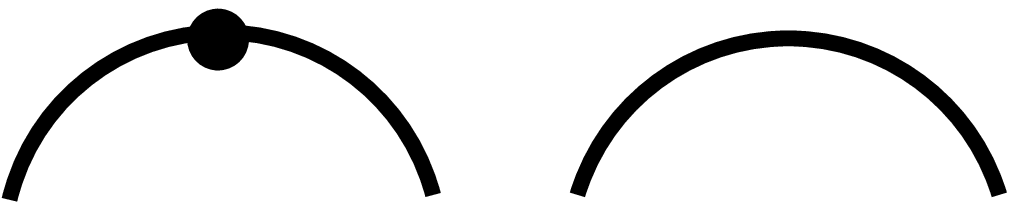} + \includegraphics[width=1in]{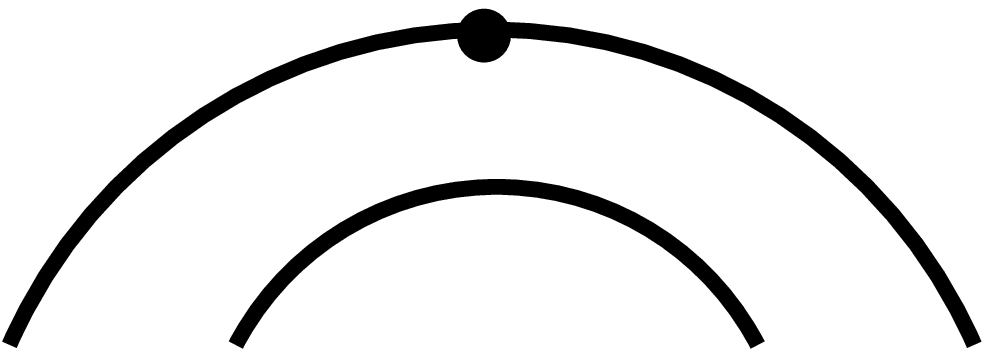}$\\
\cline{2-3} & & \hspace{1em}\\
 \small{$(j,i)$ and $(i+1,k)$ have one dot}& \includegraphics[width=1in]{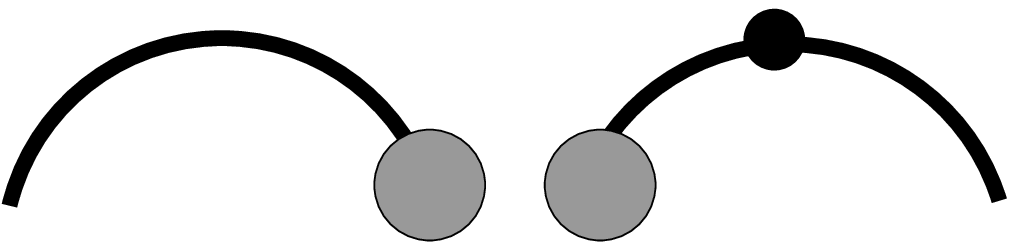} & $\includegraphics[width=1in]{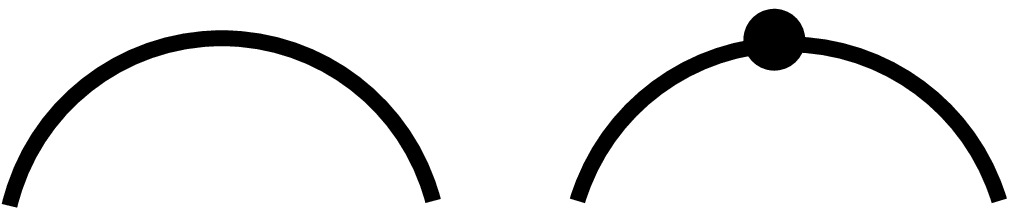} + \includegraphics[width=1in]{nest2.eps}$\\
 \cline{2-3} & & \hspace{1em}\\
&   \includegraphics[width=1in]{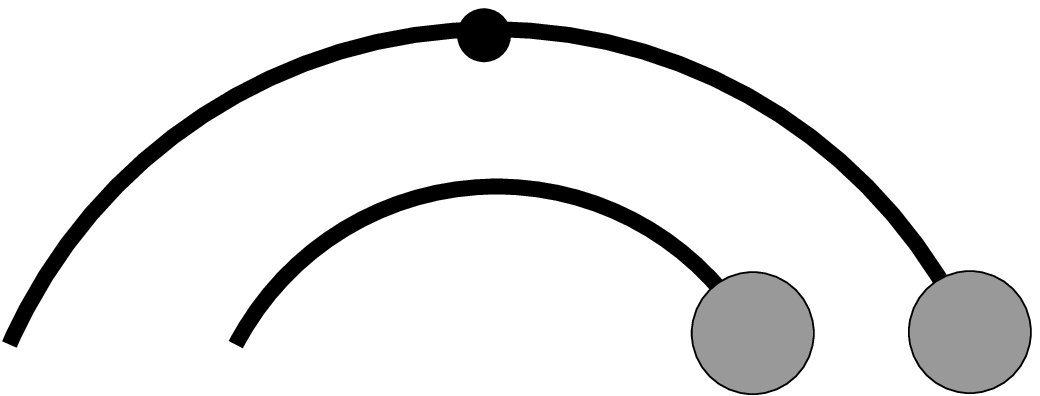} & $\includegraphics[width=1in]{nest2.eps} + \includegraphics[width=1in]{unnest1.eps}$\\
\cline{2-3} & & \hspace{1em}\\
 & \includegraphics[width=1in]{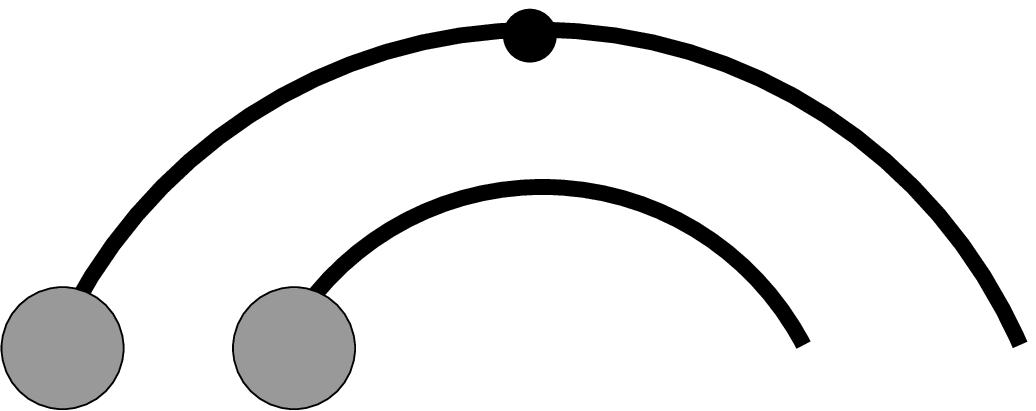} & $\includegraphics[width=1in]{nest2.eps} + \includegraphics[width=1in]{unnest2.eps}$\\
 \cline{1-3} & & \hspace{1em} \\
 {\bf  Case 4} & \hspace{.2in} \includegraphics[width=1in]{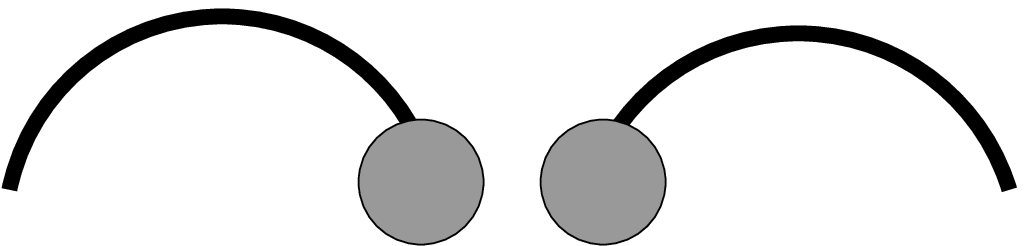}  \hspace{.2in}&  \hspace{.2in} $\includegraphics[width=1in]{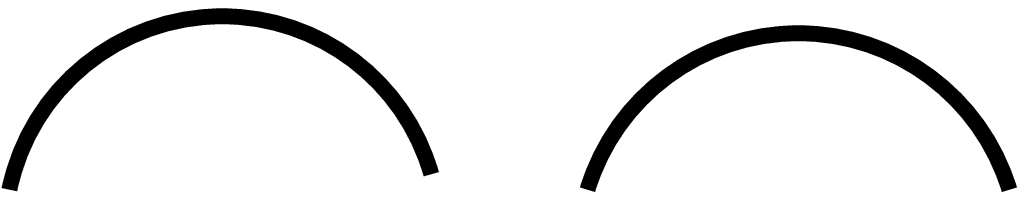} + \includegraphics[width=1in]{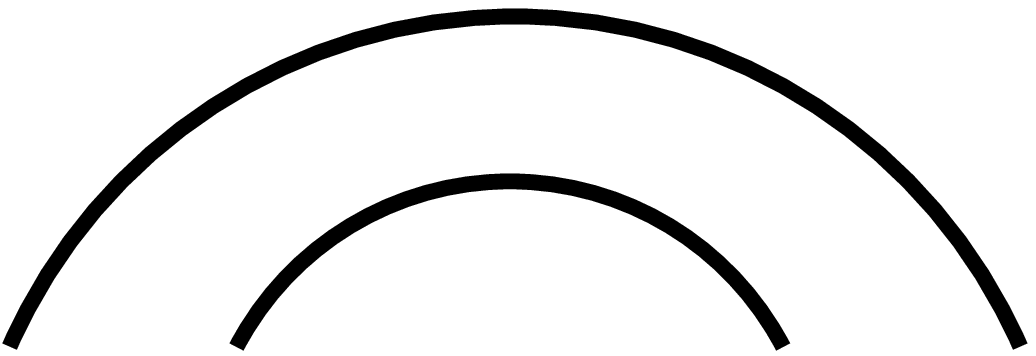}$  \hspace{.2in} \\
 \cline{2-3} & & \hspace{1em}\\
\raisebox{7pt}{\small{$(j,i)$ and $(i+1,k)$ have no dots}} & \includegraphics[width=1in]{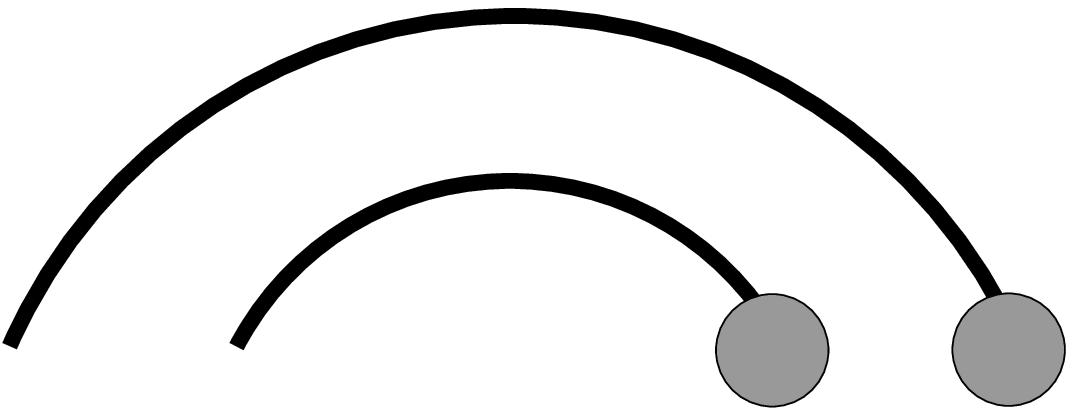} &$\includegraphics[width=1in]{nest.eps} + \includegraphics[width=1in]{revunnest0.eps}$ \\
\cline{2-3} & & \hspace{1em}\\
\raisebox{7pt}{} & \includegraphics[width=1in]{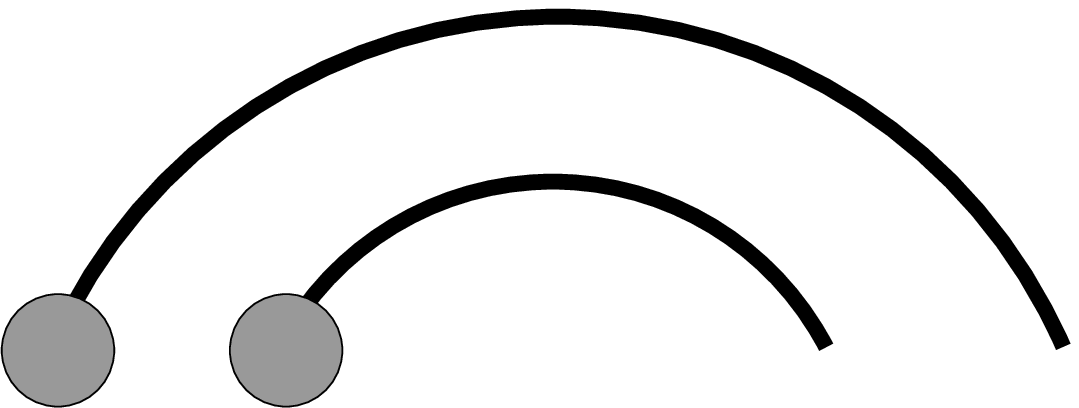} &$\includegraphics[width=1in]{nest.eps} + \includegraphics[width=1in]{revunnest0.eps}$ \\
 \hline
 \end{tabular}}
\caption{The $S_n$ - action on standard dotted noncrossing matchings. (Gray vertices indicate positions $i$ and $i+1$.)}\label{snact}
\end{figure}

\subsection{Defining a symmetric group action}

Define the map $\eta: X_{n-k,k} \rightarrow X_{n-k,n-k}$ to be $$\eta(x_1, \ldots, x_n) = (\pm p, \mp p, \ldots, p, -p, x_1, \ldots, x_n).$$ Recall that the completion map $\varphi$ from Section \ref{int} builds a matching of type $(n-k,n-k)$ from one of type $(n-k,k)$ by anchoring rays to the left of the matching. Given some noncrossing matching ${\bf a}\in B^{n-k,k}$ it is an immediate consequence of the definition of $\eta$ that $\eta (S_{\bf a}) \subset S_{\varphi({\bf a})}$.  In other words $\eta$ maps the component associated to ${\bf a}$ into the component associated to its completion. 

Recall the antipodal map $\gamma$ defined in the proof of Theorem \ref{homeopf} as $\gamma: X_{n-k,k} \rightarrow (S^2)^{n}$ by $$\gamma(x_1, \ldots, x_{n}) = (-x_1, x_2, \ldots, (-1)^n x_{n}).$$ Finally consider the map $\iota_{n}: (S^2)^n \rightarrow (S^2)^{2n-2k}$ given by $$\iota_n(x_1, \ldots, x_n) = ((-1)^{n-1}p, \ldots, (-1)^{n-1}p, (-1)^{n}x_1, \ldots, (-1)^{n}x_n).$$

\begin{lemma}\label{commute}
The following diagram is commutative.
 \[\begin{array}{rcl}
 H_*(X_{n-k,k}) & \stackrel{\eta_*}{\rightarrow} & H_*(X_{n-k,n-k}) \\
\textup{\tiny{$\gamma_*$}} \downarrow & & \downarrow {\textup{\tiny{$\gamma_*$}}} \\
H_*((S^2)^n) &  \stackrel{{\iota_{n}}_*}{\rightarrow} & H_*((S^2)^{2n-2k})\end{array}\]
\end{lemma}
\begin{proof} 
On one hand we have
\begin{eqnarray*}
\iota_{n} \circ \gamma (x_1, \ldots, x_n) &=& \iota_{n} (-x_1, x_2, \ldots, (-1)^{n-1}x_{n-1}, (-1)^n x_n) \\
&=& ((-1)^{n-1}p, \ldots, (-1)^{n-1}p, \\
&& \hspace{.5in} (-1)^{n+1}x_1, (-1)^{n+2}x_2, \ldots, (-1)^{2n-1}x_{n-1}, (-1)^{2n}x_n).
\end{eqnarray*}
And on the other hand
 \begin{eqnarray*}
 \gamma \circ \eta (x_1, \ldots, x_n) &=& \gamma ((-1)^{n-2k}p, (-1)^{n-2k-1} p, \ldots, (-1)^2p, (-1)^1p, x_1, \ldots, x_n) \\
 &=&  ((-1)^{n-2k+1}p, \ldots,(-1)^{n-2k+1}p, \\
 && \hspace{.5in}  (-1)^{n-2k+1}x_1, \ldots, (-1)^{2n- 2k-1}x_{n-1}, (-1)^{2n-2k}x_n).
 \end{eqnarray*}
 The parity of $n$ and $n-2k$ must be the same. Therefore $(-1)^{n-2k+i} = (-1)^{n+i}$ and the diagram commutes.
\end{proof}

\begin{proposition} \label{stabact}
The map $\gamma_*: H_*(X_{n-k,k}) \rightarrow H_*((S^2)^n)$ is injective, and the action of $S_n$ on $H_*((S^2)^n)$ given by permutation of coordinates stabilizes the image of $H_*(X_{n-k,k})$ under $\gamma_*$.
\end{proposition}

\begin{proof}

Viewing the maps $\eta$ and $\iota_{n}$ on diagrammatic homology generators it is straightforward to see that $\eta_*$ and ${\iota_{n}}_*$ are injective. In \cite[Corollary 3.13]{RT} we prove that $\gamma_*:H_*(X_{n-k,n-k}) \rightarrow H_*((S^2)^{2n-2k})$ is injective. Therefore by Lemma \ref{commute} we see that $\gamma_*: H_*(X_{n-k,k}) \rightarrow H_*((S^2)^n)$ is injective and the first half of the lemma is proven.

Consider the subgroup $G\leq S_{2n-2k}$ fixing the numbers $1, \ldots, n-2k$. Then $G$ is isomorphic to $S_n$ via the map $(i \hspace{.1in} j) \mapsto (i+n-2k \hspace{.1in} j+n-2k)$. The action of $G$ on $H_*(X_{n-k,n-k})$ is given by pulling back the action of $G$ on $H_*((S^2)^{2n-2k})$. Since the homology of a cartesian product of copies of $S^2$ is a tensor product of copies of $H_*(S^2)$, the symmetric group action that permutes tensor factors commutes with the inclusion map ${\iota_n}_*$.  We can therefore prove the second part of the lemma by showing that the action of $G$ on $H_*(X_{n-k,n-k})$ stabilizes the image of $H_*(X_{n-k,k})$ under $\eta_*$.

Again note that given ${\bf a}\in B^{n-k,k}$ we have $\eta (S_{\bf a}) \subset S_{\varphi({\bf a})}$. On the level of homology $\eta$ sends generators in $H_*(X_{n-k,k})$ to their associated completions in $H_*(X_{n-k,n-k})$. The new dotted arcs in the completed generators will have right endpoint at the vertex where the ray was formerly incident and left endpoint on one of the first $n-2k$ vertices. While this new dotted noncrossing matching is not necessarily standard, it is equivalent to a unique standard noncrossing matching via Type II relations. 

Diagrammatically, the image $\eta_*(H_*(X_{n-k,k}))$ is exactly those generators in $H_*(X_{n-k,n-k})$ with no undotted arcs incident on the first $n-2k$ vertices. Indeed the preimage of such a generator under $\eta_*$ is the standard dotted noncrossing matching with the first $n-2k$ vertices removed,  identical undotted arcs, and the unique standard arrangement of rays and dotted arcs in the remaining positions.

In order to prove that the subgroup action of $G$ stabilizes $\eta_*(H_*(X_{n-k,k}))$ we need to show that each simple transposition $(i\hspace{.1in} i+1)$ where $n-2k+1\leq i<2n-2k$ maps elements of  $\eta_*(H_*(X_{n-k,k}))$ to a linear combination of elements of $H_*(X_{n-k,n-k})$ with no undotted arcs incident on the first $n-2k$ vertices.  Examining rows 1, 2, 4, and 7 on the chart in Figure \ref{snact} we see that it is not possible for such a simple transposition to map a generator without an undotted arc incident on the first $n-2k$ vertices to one with an undotted arc incident on the first $n-2k$ vertices.
\end{proof}

\begin{theorem}
There is a graded $S_n$ - action on $H_*(X_{n-k,k})$ described by the chart in Figure \ref{snact} together with the chart in Figure \ref{rayact}.
\end{theorem}
\begin{proof}
Using the commutative diagram in Lemma \ref{commute} together with the chart in Figure \ref{snact} we are able to calculate $\gamma_*^{-1}((i \hspace{.2in} i+1)\cdot \gamma_*(M))$ where $M\in H_*(X_{n-k,k})$. By Proposition \ref{stabact} we have a well-defined $S_n$ - action given by  $(i \hspace{.2in} i+1)\cdot M := \gamma_*^{-1}((i \hspace{.2in} i+1)\cdot \gamma_*(M)) $.
\end{proof}

\begin{figure}[h]
\scalebox{.75}{
\begin{tabular}{|c||c|c|}
\cline{1-3} & &   \hspace{1em} \\
 {\bf Vertex Labelings} & $M$ & $(i~~~~ i+1)\cdot M$  \\
 & &   \hspace{1em} \\
 \cline{1-3} & &  \hspace{1em} \\
 \raisebox{10pt}{{\bf Case 5}} &\includegraphics[width=.3in]{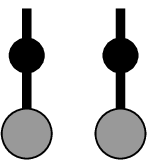}& \includegraphics[width=.3in]{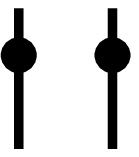}\\
 \raisebox{7pt}{\small{$i,i+1$ both on rays}}&  &  \\
 \cline{1-3} & & \hspace{1em}\\
 \raisebox{10pt}{{\bf  Case 6}} & \includegraphics[width=.8in]{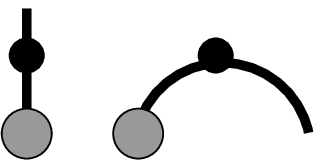} &   \includegraphics[width=.8in]{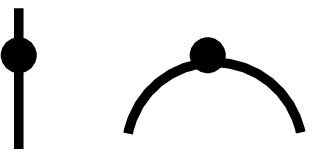} \\
 \raisebox{7pt}{\small{$i,i+1$ on ray and dotted arc}} & & \\
 \cline{1-3} & & \hspace{1em}\\
  {\bf  Case 7}  &\includegraphics[width=.8in]{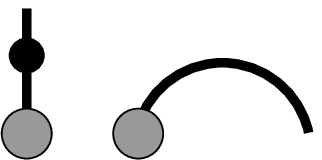} & $\includegraphics[width=.8in]{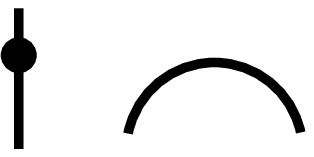} \hspace{.2in} \raisebox{10pt}{$+$} \hspace{.2in}\includegraphics[width=.8in]{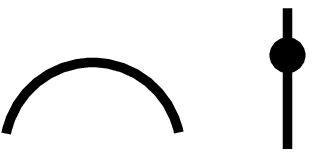}$\\
\cline{2-3} & & \hspace{1em}\\
 \small{$i, i+1$ on ray and undotted arc}& \includegraphics[width=.8in]{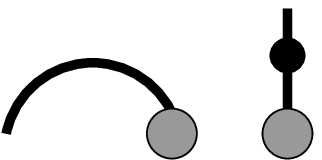} & $\includegraphics[width=.8in]{c72.eps} \hspace{.2in} \raisebox{10pt}{$+$} \hspace{.2in} \includegraphics[width=.8in]{c71.eps}$\\
 & & \hspace{1em}\\
 \hline
 \end{tabular}}
\caption{The $S_n$ - action on matchings of type $(n-k,k)$ with rays. (Gray vertices indicate positions $i$ and $i+1$.)}\label{rayact}
\end{figure}

\subsection{Verifying the Springer action}

Young tabloids are equivalence classes of Young tableaux where the equivalence relation is permutation of numbers within rows. The symmetric group $S_n$ acts on tabloids; indeed given $\sigma \in S_n$ and a tabloid $T$ we let $\sigma\cdot T$ be the tabloid with $\sigma(i)$ replacing $i$ for all $i = 1, \ldots, n$.  Consider the vector space $U_{\lambda}$ of complex linear combinations of tabloids of shape $\lambda$. For a tabloid $T$ write its corresponding element in $U_{\lambda}$ as $v_T$.  Then the action of $S_n$ on tabloids induces an action on $U_{\lambda}$.

Consider $0\leq m\leq \lfloor {n/2} \rfloor$ and $\lambda = (n-m, m)$ a partition of the number $n$. 
For each tableau $T$ let $\textup{Col}(T)$ be the subgroup of permutations in $S_n$ that stabilize all columns of $T$. For each tableau $T$ define the following vector in $U_{n-m,m}$.
$$e_{T} = \sum_{\sigma\in \textup{Col}(T)} \textup{sign}(\sigma) v_{\sigma \cdot T}$$
The subspace of $U_{n-m,m}$ generated by these vectors which we denote $V_{n-m,m}$ is known as the Specht module. A classical result says that $V_{n-m,m}$ is the irreducible representation of $S_n$ corresponding to the partition $(n-m,m)$. Furthermore $V_{n-m,m}$ has a basis given by $\{ e_T\}$ where $T$ are standard tableaux.  

Given a standard dotted noncrossing matching $M$ of type $(n-k,k)$ with $m$ undotted arcs let $\textup{Undot}(M)$ be the subgroup of $S_n$ generated by all transpositions $(i \hspace{.1in} j)$ where $(i,j)$ is an undotted arc in $M$. Let $e_M$ be the vector in $U_{n-m,m}$ defined as
$$e_M = \sum_{\sigma\in \textup{Undot}(T)} \textup{sign}(\sigma) v_{\sigma \cdot h(M)}.$$ We call the subspace of $U_{n-m,m}$ generated by the vectors $e_M$ the matching module. We denote the matching module by $W_{n-m,m}$. Define a map $\zeta: H_{2m}(X_{n-k,k}) \rightarrow U_{n-m,m}$ by $\zeta(M) = e_M$. Then the image of $\zeta$ is exactly the matching module.

Define the map $f: U_{n-m,m} \rightarrow U_{2n-2k-m,m}$ as $f(v_T) = v_{T'}$ where $T'$ is the tabloid with numbers $1$ through $n-2k$ in the top row and $i+n-2k$ in the row where $i$ appears in $T$. Note that $f$ is injective. Figure \ref{fmap} has an example.

\begin{figure}[h]
$T = \begin{tabular}{|c|c|c|c|c|}
\cline{1-5} 1 & 4 & 5 & 6 & 7 \\
\cline{1-5} 2 & 3 & \multicolumn{2}{c}{} \\
\cline{1-2} \multicolumn{4}{c}{} \vspace{-1.3em} \end{tabular} \hspace{.75in} T' = \begin{tabular}{|c|c|c|c|c|c|}
\cline{1-6} 1 & 2 & 5 & 6 & 7 & 8 \\
\cline{1-6} 3 & 4 & \multicolumn{2}{c}{} \\
\cline{1-2} \multicolumn{4}{c}{} \vspace{-1.3em} \end{tabular} $
\caption{The behavior of tabloids under the map $f$ where $n = 7$ and $k = 3$.}\label{fmap}
\end{figure}

\begin{lemma}\label{dotinfo}
Let $M$ be a standard dotted noncrossing matching. Then $f(e_M) = e_{\varphi (M)}$ where $\varphi(M)$ is understood to be the completion of the associated matching along with dotting information.
\end{lemma}
\begin{proof}
Recall that $h$ is the bijection taking standard Young tableaux to standard dotted noncrossing matchings. If $(i,j)$ is an undotted arc in $M$ then $(i+n-2k, j+n-2k)$ is an undotted arc in $\varphi(M)$. This means that $h(M)' = h(\varphi(M))$. Furthermore if $\textup{Undot}(M)$ is generated by transpositions $(i \hspace{.1in} j)$ then $\textup{Undot}(\varphi(M))$ is generated by transpositions $(i+n-2k \hspace{.1in} j+n-2k)$. 

The bijection $(i \hspace{.1in} j) \mapsto (i+n-2k \hspace{.1in} j+n-2k)$ between $S_n$ and $G$ restricts to a bijection between $\textup{Undot}(M)$ and $\textup{Undot}(\varphi(M))$. Moreover for $\sigma\in \textup{Undot}(M)$ and its corresponding $\sigma'\in \textup{Undot}\varphi(M)$ we have $(\sigma\cdot \theta(M))' = \sigma' \cdot \theta(\varphi(M))$. Thus we conclude that $f(e_M) = e_{\varphi(M)}$. 
\end{proof}

\begin{corollary}
The following diagram is commutative
 \[\begin{array}{rcl}
 H_*(X_{n-k,k}) & \stackrel{\eta_*}{\rightarrow} & H_*(X_{n-k,n-k}) \\
\textup{\tiny{$\zeta$}} \downarrow & & \downarrow {\textup{\tiny{$\zeta$}}} \\
U_{n-m,m} &  \stackrel{f}{\rightarrow} & U_{2n-2k-m,m}\end{array}\]
\end{corollary}

\begin{lemma} \label{fequiv}
The action of $S_n$ on $U_{n-m,m}$ commutes with the action of $G$ on $U_{2n-2k-m,m}$.
\end{lemma}

\begin{proof}
Let $\sigma'$ be the element of $G$ corresponding to $\sigma \in S_n$ via the bijection described above. For $v_T\in U_{n-m,m}$ we want to show that $f(\sigma\cdot v_T) = \sigma ' \cdot f(v_T) = v_{\sigma '\cdot T'}$.

The tabloid  $\sigma\cdot T$ has $\sigma(i)$ in the row that $i$ occupied in $T$. The tabloid corresponding to $f(\sigma\cdot v_T)$ has $1, \ldots, n-2k$ in the top row and $\sigma(i) + n-2k$ in the row that $i$ occupied in $T$. On the other hand $T'$ has $1,\ldots, n-2k$ in the top row and $i+n-2k$ in the row that $i$ occupied in $T$. The tabloid $\sigma'\cdot T'$ has $1, \ldots, n-2k$ on the top row and $\sigma'(i+n-2k)$ in the row that $i$ occupied in $T$. Since $\sigma'(i+n-2k) = \sigma(i) + n-2k$ we conclude that $f(\sigma\cdot v_T) = \sigma' \cdot f(v_T)$.
\end{proof}

\begin{corollary}
The map $\zeta: H_{2m}(X_{n-k,k}) \rightarrow W_{n-m,m}$ is an $S_n$-equivariant isomorphism of  complex vector spaces.
\end{corollary}
\begin{proof}
As we noted earlier $\eta_*$ is injective, so $\eta_*: H_{2m}(X_{n-k,k}) \rightarrow \eta_*(H_{2m}(X_{n-k,k})$ is an isomorphism. In \cite[Lemma 4.2]{RT} we prove that $\zeta: H_{2m}(X_{n-k,n-k}) \rightarrow W_{2n-2k-m,m}\subset U_{2n-2k-m,m}$ is an $S_{2n-2k}$ - equivariant isomorphism, so it follows that $\zeta: \eta_*(H_{2m}(X_{n-k,k})) \rightarrow f(W_{n-m,m})$ is a $G$ - equivariant isomorphism. Finally Lemma \ref{dotinfo} implies that $f$ is injective, so $f: W_{n-m,m} \rightarrow f(W_{n-m,m})$ is an isomorphism. Thus $\zeta = f^{-1} \circ \zeta \circ \eta_*: H_{2m}(X_{n-k,k}) \rightarrow W_{n-m,m}$ is an isomorphism.

By construction the action of $S_n$ on $H_{2m}(X_{n-k,k})$ commutes with the action of $G$ on $\eta_*(H_{2m}(X_{n-k,k}))$.  Since $\zeta: H_{2m}(X_{n-k,n-k}) \rightarrow U_{2n-2k-m,m}$ is $S_{2n-2k}$ - equivariant it is also $G$ - equivariant. This information together with Lemma \ref{fequiv} proves that $\zeta: H_{2m}(X_{n-k,k}) \rightarrow W_{n-m,m}$ is $S_n$-equivariant.
\end{proof}

\begin{theorem} \label{springer repn}
The Specht module $V_{n-m,m}$ and the matching module $W_{n-m,m}$ are equal as $S_n$-representations.  Thus the $S_n$ action on $H_*(X_{n-k,k})$ is the Springer representation.
\end{theorem}

\begin{proof} Since $V_{n-m, m}$ and $W_{n-m, m}$ are $S_n$-representations in $U_{n-m,m}$ the intersection $V_{n-m,m}\cap W_{n-m,m}$ is also an $S_n$ representation.  The Specht module $V_{n-m,m}$ is irreducible, so $V_{n-m,m}\cap W_{n-m,m}$ is either $0$ or $V_{n-m,m}$. 
                                                                               
Let $M$ be the standard dotted noncrossing matching with $m$ undotted arcs 

\noindent $(1,2), (3,4), \ldots, (2m-1, 2m)$. Then $M$ represents a basis element in $H_{2m}(X_{n-k,k})$. Furthermore the columns of $h(M)$ are exactly the undotted arcs of $M$, so $\textup{Undot}(M) = \textup{Col}(h(M))$. This means that $e_{M} = e_{h(M)}$, so the matching module and Specht module always have a common vector.

Since $V_{n-m,m} \cap W_{n-m,m}$ is nonempty we conclude that the intersection is all of $V_{n-m,m}$.  Because $V_{n-m,m}$ and $W_{n-m,m}$ have the same dimension, we conclude that $V_{n-m,m} = W_{n-m,m}$.  Since this is the irreducible representation corresponding to partition $(n-m,m)$, Proposition \ref{GP} allows us to conclude that this is the Springer representation.
\end{proof}

\subsection{A skein-theoretic formulation of the Springer action}
The Springer action on $H_*(X_{n-k,k})$ is described diagrammatically in the charts in Figures \ref{snact} and \ref{rayact}. We can also describe the action skein-theoretically. 

\begin{theorem}
Given standard dotted noncrossing matching $M\in H_*(X_{n-k,k})$ and $\sigma\in S_n$ glue a flattened braid corresponding to $\sigma$ to the bottom of $M$ forming $M'$. 
Then the Springer action $\sigma\cdot M$ is equal to $s(M')$ where $s$ is defined as follows.
\begin{itemize}
\item{$s\left( \raisebox{-.13in}{\includegraphics[width=.3in]{crossing.eps}}\right) = s\left( \raisebox{-.13in}{\includegraphics[width=.3in]{smooth1.eps}}\right) + s \left( \raisebox{-.13in}{\includegraphics[width=.3in]{smooth2.eps}}\right)$}
\item{$s\left( \raisebox{-.15in}{\includegraphics[width=.3in]{doubdot.eps}}\right) = 0$}
\item{ $s\left( M' \sqcup \raisebox{-.1in}{\includegraphics[width=.3in]{simcirc.eps}} \right) = s(-2M')$}
\item{$s\left( M' \sqcup \raisebox{-.1in}{\includegraphics[width=.3in]{simcircdot.eps}} \right) = s(M')$}
\end{itemize}
\end{theorem}

As an example consider the permutation $(1 \hspace{.1in} 2 \hspace{.1in} 3)$ acting on the following generator of $H_2(X_{2,1})$.

$$(1 \hspace{.05in} 2 \hspace{.05in} 3) \cdot \raisebox{-.05in}{\includegraphics[width=.4in]{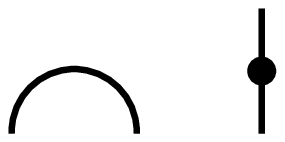}} \hspace{.1in} = 
s\left( \raisebox{-.2in}{\includegraphics[width=.4in]{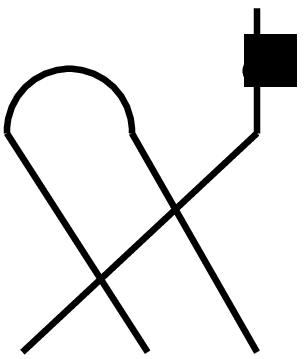}}\right)  = s\left(\raisebox{-.2in}{\includegraphics[width=.4in]{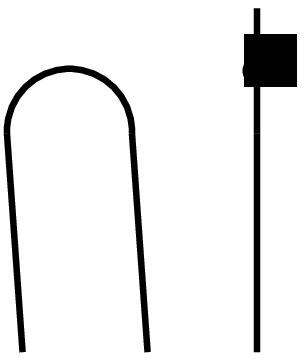}}\right) + s\left(\raisebox{-.2in}{\includegraphics[width=.4in]{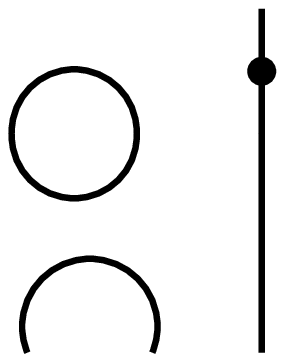}}\right) + s\left(\raisebox{-.2in}{\includegraphics[width=.4in]{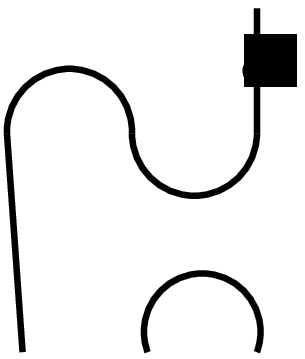}}\right) + s\left(\raisebox{-.2in}{\includegraphics[width=.4in]{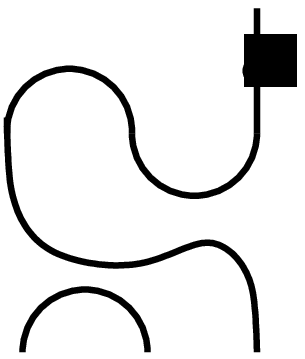}}\right)$$ 
$$= \raisebox{-.05in}{\includegraphics[width=.4in]{skeinex1.eps}} \hspace{.1in} - \hspace{.1in} \raisebox{-.05in}{2 \includegraphics[width=.4in]{skeinex1.eps}} \hspace{.1in} + \hspace{.1in} \raisebox{-.05in}{\includegraphics[width=.4in]{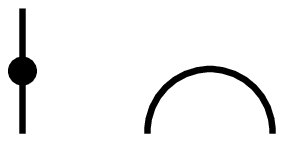}} \hspace{.1in} + \hspace{.1in} \raisebox{-.05in}{\includegraphics[width=.4in]{skeinex1.eps}} =  \raisebox{-.05in}{\includegraphics[width=.4in]{skeinex7.eps}} $$

\section{Appendix}
In Section 3 we cite Propositions \ref{CircleDist} and \ref{Khovarrows} from \cite{K}. Proofs of these results are left to the reader in \cite{K}. For the sake of completeness we include proofs of these statements here.

\begin{proof}[{\bf Proof of Proposition \ref{CircleDist}}]
Let $|{\bf a}w({\bf b})|$ be denoted by $c$. This proof proceeds by induction on $c$. We begin with the case that $c=1$ and then argue by induction. The following algorithm finds a sequence of $\frac{n}{2}-1$ moves on ${\bf a}$ that yields ${\bf b}$. An arc $(i,j)$ in ${\bf b}$ is said to be unpaired in ${\bf a}w({\bf b})$ if ${\bf a}$ does not have arc $(i,j)$. Let ${\bf a} = {\bf a_0}$ and let $t = 0$.
\begin{enumerate}
\item If ${\bf b}$ has no unpaired arcs in ${\bf a_t}w({\bf b})$, then ${\bf a_t} = {\bf b}$ and we are done. Otherwise begin with the narrowest leftmost arc $(i,j)$ in ${\bf b}$ that is unpaired in ${\bf a_t}w({\bf b})$. Then ${\bf a_t}$ has arcs $(i,k)$, $(j,l)$ where $i,j,k,l$ have no particular order. 
\item Perform $\rightarrow$ or $\leftarrow$ in ${\bf a_t}$ to produce a new matching ${\bf a_{t+1}}$ with arcs $(i,j)$ and $(k,l)$. 
\item Increment $t$ by 1, and repeat the first two steps.
\end{enumerate}

It is always possible to perform $\rightarrow$ or $\leftarrow$ as required.  At each step the algorithm guarantees all arcs nested beneath $(i,j)$ are paired in ${\bf a_t}w({\bf b})$ since arcs beneath $(i,j)$ are narrower than $(i,j)$. Thus each arc in ${\bf a_t}$ either has zero or two endpoints between vertices $i$ and $j$. This means that the arcs $(i,k)$ and $(j,l)$ are adjacent in ${\bf a_t}$, and $\rightarrow$ or $\leftarrow$ can be performed to get a new matching ${\bf a_{t+1}}$ with arcs $(i,j)$ and $(k,l)$.

The number of circles in ${\bf a_{t+1}}w({\bf b})$ is one more than the number of circles in ${\bf a_{t}}w({\bf b})$. To see this note that the arcs $(i,k), (j,l)$ in ${\bf a_t}$ and the arc $(i,j)$ in ${\bf b}$ are all part of the same circle in ${\bf a_t}w({\bf b})$. The move that produces ${\bf a_{t+1}}$ must change the number of circles. Since this move acts on two arcs that are part of a single circle, the only possibility is that the move increases the number of circles by one.

Since we increase the number of circles with each iteration of our algorithm, this process terminates after $\frac{n}{2}-1$ iterations at which point we will have $\frac{n}{2}$ circles - the maximum number of circles possible. This proves that $d({\bf a},{\bf b})\leq \frac{n}{2}-1$. 

In an arbitrary sequence of moves, each time ${\bf a_t} \rightarrow {\bf a_{t+1}}$ or ${\bf a_t}\leftarrow {\bf a_{t+1}}$ is performed the number of circles in ${\bf a_t}w({\bf b})$ is either one greater or one less than ${\bf a_{t+1}}w({\bf b})$ depending on whether the move taking ${\bf a_t}$ to ${\bf a_{t+1}}$ joins two circles or splits a single circle. Since we start with one circle, it will require at least $\frac{n}{2}-1$ moves to produce $\frac{n}{2}$ circles. This means our algorithm generates a minimal sequence $({\bf a}= {\bf a_0}, {\bf a_1}, \ldots, {\bf a_{\frac{n}{2}-2}}, {\bf a_{\frac{n}{2}-1}} = {\bf b})$ and $d({\bf a},{\bf b}) = \frac{n}{2}-1$. 

Now consider ${\bf a},{\bf b}\in B^{n/2, n/2}$ where ${\bf a}w({\bf b})$ has $c$ circles, and enumerate the circles $\alpha_1, \ldots, \alpha_c$. Each circle $\alpha_r$ passes through the horizontal axis $2q_r$ times.  By the above argument we have
$$d({\bf a},{\bf b}) = \sum_{1\leq r \leq c} (q_r - 1) = \left( \sum_{1\leq r \leq c} q_r\right) - c = \frac{n}{2}-c.$$
\end{proof}

\begin{proof}[{\bf Proof of Proposition \ref{Khovarrows}}]

This is proven using induction on the distance between ${\bf a}$ and ${\bf b}$. If there exists a minimal sequence of moves taking ${\bf a}$ to ${\bf b}$ that begins with $\leftarrow$ the inductive step is clear. The lemmas below are required in order to handle the case that no minimal sequences of moves taking ${\bf a}$ to ${\bf b}$ begin with $\leftarrow$. When this occurs, the following arguments prove that there always exists a minimal sequence for ${\bf a}$ and ${\bf b}$ of the form $(\rightarrow, \rightarrow, \ldots, \rightarrow, \rightarrow)$. In other words, we want to show that in this case we take ${\bf a}={\bf c}$.

\begin{lemma}
Let ${\bf a}$ be the matching with arcs $$(1,2), (3,4), \ldots, (n-3, n-2), (n-1, n)$$ and ${\bf b}$ be the matching with arcs $$(2,3), (4,5), \ldots, (n-4, n-3), (n-2, n-1), (1, n).$$ 
Then there exists a minimal sequence $({\bf a} \rightarrow {\bf a_1} \ldots {\bf a_{m-1}} \rightarrow {\bf b})$.
\end{lemma}
\begin{proof}
Consider the minimal sequence generated by the algorithm in the proof of Proposition \ref{CircleDist}. Each move takes the matching with arcs $(1,i), (i+1, i+2)$ to the matching with arcs $(1,i+2), (i,i+1)$. Since arcs $(1,i), (i+1,i+2)$ are unnested these moves always have the form $\rightarrow$.  
\end{proof}

This next lemma addresses circles of this type within more complicated gluings of two matchings.

\begin{lemma}\label{nesty}
Let $i_1<\cdots <i_{p}$. Let ${\bf a}$ be some matching with arcs $(i_1, i_2), \ldots, (i_{p-1}, i_{p})$. Let ${\bf b}$ be some matching with arcs $(i_2, i_3), \ldots, (i_{p-2}, i_{p-1}), (i_1, i_{p})$. There exists a minimal sequence for ${\bf a}$ and ${\bf b}$  where all moves on arcs incident on the vertices $i_1, \ldots, i_{p}$ have the form $\rightarrow$.
\end{lemma}
\begin{proof}
We prove this using induction. Assume $i_1<i_2<i_3<i_4$ and that ${\bf a}$ is a matching with arcs $(i_1, i_2), (i_3,i_4)$ while ${\bf b}$ is a matching with arcs $(i_2, i_3), (i_1,i_4)$. The algorithm given in the proof of Proposition \ref{CircleDist} will eventually find arc $(i_2,i_3)$ unpaired in ${\bf a_t}w({\bf b})$. It will perform the move $\rightarrow$  taking the unnested pair of arcs $(i_1, i_2), (i_3, i_4)$ to the nested pair of arcs $(i_2,i_3), (i_1,i_4)$.

Assume that the above result is true for some $p$. Consider $$i_1<\cdots <i_{p+2}.$$ Take ${\bf a}$ and ${\bf b}$ to be matchings as in the statement of the corollary. The algorithm from the proof of Proposition \ref{CircleDist} will eventually find the narrowest arc in ${\bf b}$ incident on these vertices. Say this arc is $(i_k, i_{k+1})$. The move $\rightarrow$ will be performed on matching ${\bf a_t}$ producing new arcs $(i_{k-1}, i_{k+2}), (i_{k}, i_{k+1})$. The arc $(i_k, i_{k+1})$ is now paired in ${\bf a_{t+1}}w({\bf b})$. After renumbering the remaining arcs are incident on $p$ vertices as in the statement of the corollary. We apply the inductive hypothesis to the remaining arcs, and the result follows by induction.
\end{proof}

Before the next lemma, we provide one technical definition that will make the statement of the lemma easier.

\begin{definition}
Let $i<j<k<l$ and suppose $(i,l), (j,k)$ are nested arcs in some matching ${\bf a}$. We say that an arc is between $(i,l)$ and $(j,k)$ if the arc in ${\bf a}$ with left endpoint between $i$ and $j$ and right endpoint between $k$ and $l$. 
\end{definition}

\begin{lemma}
Let ${\bf a},{\bf b}\in B^{n/2, n/2}$. Let $i<j<k<l$ and suppose $(i,l), (j,k)$ are nested arcs in ${\bf a}$ that are part of the same circle in ${\bf a}w({\bf b})$ such that there is no arc in ${\bf a}$ lying between $(i,l)$ and $(j,k)$ that is also part of that circle. Then the number of arcs between $(i,l), (j,k)$  is even.
\end{lemma}

\begin{proof}
Given some circle $c$ in ${\bf a}w({\bf b})$ let $W(c, z)$ be the winding number of $c$ about $z$.  Consider a line segment between $z_0, z_1\notin c$ which is transverse to $c$. Say that the number of intersections of that segment with the circle $c$ is some number $r$. Then mod 2 intersection theory tells us that
$$ W(c, z_0) \equiv W(c, z_1) + r \textup{ (mod } 2).$$

Take some arc between $(i,l)$ and $(j,k)$. Say this arc is part of circle $c$ in ${\bf a}w({\bf b})$. Note that the circle $c$ is not the same circle that $(i,l)$ and $(j,k)$ are part of. Vertices $i$ and $j$ are both either inside or outside the circle $c$. Consider the vertices $i$ and $j$ to be points in the plane. Then
$$W(i,c) \equiv W(j,c) \textup{( mod} 2).$$ Putting this together with the previous statement, say the segment between $i$ and $j$ intersects $c$ exactly $r$ times. Then
$W(i,c) \equiv W(j,c) + r \textup{( mod} 2)$ and so $r \equiv 0 \textup{( mod} 2)$. We conclude that $c$ has an even number of intersections with the line segment between $i$ and $j$. 

Each arc in $c$ with both endpoints between $i$ and $j$ contributes an even number of intersections, so the number of arcs between $(i,l)$ and $(j,k)$ must also be even. Since this is true for each circle with some arc between $(i,l)$ and $(j,k)$, there is an even number of arcs between $(i,l)$ and $(j,k)$. 
\end{proof}

\begin{lemma}\label{nestarcs}
Within the collection of arcs between $(i,l)$ and $(j,k)$ there are two adjacent arcs that are part of the same circle.
\end{lemma}

\begin{proof}
By the above argument there is an even number of vertices between $i$ and $j$. Given two circles $c$ and $c'$ with arcs that lie between $(i,l)$ and $(k,j)$ the argument in the previous proof says that there are an even number of arcs from $c'$ between any pair of arcs in $c$. 

Because there are finitely many vertices between $i$ and $j$, there must be some pair of arcs from a single circle that have no arcs from any other circle between them. Otherwise we could find infinitely many arcs between $(i,l)$ and $(j,k)$.
\end{proof}

\begin{lemma}\label{nestonly}
Given ${\bf a},{\bf b}\in B^{n/2, n/2}$ if every minimal sequence $({\bf a},{\bf  a_1}, \ldots, {\bf a_{m-1}}, {\bf b})$ begins with ${\bf a}\rightarrow {\bf a_1}$ then there exists a minimal sequence $(\rightarrow,\rightarrow, \cdots, \rightarrow, \rightarrow)$. 
\end{lemma}

\begin{proof}

Assume that all minimal sequences for ${\bf a}$ and ${\bf b}$ begin with $\rightarrow$. In particular this means that there does not exist a pair of nested arcs in ${\bf a}$ that are part of the same circle in ${\bf a}w({\bf b})$ and have no other arc from that same circle between them. (The existence of such a pair would violate our assumption that all minimal sequences begin with $\rightarrow$.)

Say there is some pair of nested arcs $(i,l), (j,k) \in {\bf a}$ that are part of the same circle in ${\bf a}w({\bf b})$. By Lemma \ref{nestarcs} there must be a pair of nested, adjacent arcs in ${\bf a}$ contributed by some other circle of ${\bf a}w({\bf b})$ lying between $(i,l)$ and $(j,k)$. In this case, we could perform $\leftarrow$ on this pair once again violating our assumption that all minimal sequences must begin with $\rightarrow$. We conclude that the no pair of nested arcs in ${\bf a}$ are part of the same circle in ${\bf a}w({\bf b})$.

Since we have just shown the arcs in each circle in ${\bf a}w({\bf b})$ are pairwise unnested, they have the form 
$$(i_1, i_2), (i_3,i_4), \ldots, (i_{t-1}, i_t) \textup{ where } i_1<i_2<\cdots <i_{t-1}<i_t.$$ The arcs in ${\bf b}$ that comprise the rest of this circle necessarily have the form $$(i_2, i_3), \ldots, (i_{t-2}, i_{t-1}), (i_1, i_t).$$ 
Consider the minimal sequence of moves for ${\bf a}$ and ${\bf b}$ obtained by the algorithm in the proof of Proposition \ref{CircleDist}. By Lemma \ref{nesty} each move has the form $\rightarrow$.  Therefore we have a minimal sequence  consisting only of $\rightarrow$ moves, and the lemma is proven.
\end{proof}

Having established Lemma \ref{nestonly} we can now easily finish the proof of Proposition \ref{Khovarrows}.

We proceed using induction. If the distance between matchings is 1, the result is clear since there is only one move between matchings in any minimal sequence. 

Now assume the lemma is true for all ${\bf a},{\bf b}$ with $d({\bf a},{\bf b}) \leq m-1$. Assume that $d({\bf a},{\bf b}) = m$. Let $({\bf a} = {\bf a_0}, {\bf a_1}, \ldots, {\bf a_{m-1}}, {\bf a_m} = {\bf b})$ be a minimal sequence of moves from ${\bf a}$ to ${\bf b}$ such that the move $\leftarrow$ occurs as early as possible. 

If ${\bf a_0} \leftarrow {\bf a_1}$ we can apply the inductive assumption to get ${\bf a_1}\succ {\bf c}\prec {\bf a_m}$ with $d({\bf a_1}, {\bf a_m}) = d({\bf a_1},{\bf c}) + d({\bf c}, {\bf a_m})$ and $({\bf a_1} \leftarrow \cdots \leftarrow {\bf c} \rightarrow \cdots \rightarrow {\bf a_m})$. This proves the lemma in this case since we have a minimal sequence of the form $({\bf a} \leftarrow {\bf a_1}\leftarrow \cdots \leftarrow {\bf c} \rightarrow \cdots \rightarrow {\bf b})$.

If ${\bf a_0}\rightarrow {\bf a_1}$ then there does not exist a minimal sequence for ${\bf a}$ and ${\bf b}$ with first move $\leftarrow$. By Lemma \ref{nestonly} there exists a minimal sequence of the form $({\bf a}\rightarrow \cdots \rightarrow {\bf b})$ in this case. 
\end{proof}

\end{document}